\documentclass[12pt,a4paper]{amsart} %a4paper toegevoegd door Johan
 \usepackage{geometry} %toegevoegd door Johan
\usepackage{amscd,amssymb,amsthm,amsmath,amssymb,textcomp,supertabular,longtable,enumerate,rotating,mathrsfs,mathtools,hyperref,latexsym,amscd,amsbsy,mathrsfs,tikz-cd}
\usepackage{pdflscape}
\usepackage[matrix,arrow,curve]{xy}
\usepackage{todonotes}
\newtheorem{theorem}[equation]{Theorem}

\newtheorem{proposition}[equation]{Proposition}
\newtheorem{lemma}[equation]{Lemma}

\theoremstyle{definition}

\newtheorem{definition}[equation]{Definition}

\theoremstyle{remark}
\newtheorem{remark}[equation]{Remark}

\makeatletter\@addtoreset{equation}{section} \makeatother

\swapnumbers

\newtheoremstyle{dotless}{}{}{\rm}{}{\sc}{}{ }{}

\theoremstyle{dotless}

\author{Karim R\'ega}

\title{Moduli of Anti-Invariant Higgs Bundles}

\address{\emph{Karim R\'ega}
\newline
\textnormal{School of Mathematics and Maxwell Institute, The University of Edinburgh,  Edinburgh EH9 3JZ, UK.}
\newline
\textnormal{\texttt{karim.rega@ed.ac.uk}}}

\newcommand{\A}{\mathbb{A}}
\newcommand{\AC}{\mathbb{A}^1 \times C}
\newcommand{\C}{\mathbb{C}}

\newcommand{\Z}{\mathbb{Z}}
\newcommand{\Gm}{\mathbb{G}_m}
\newcommand{\End}{\text{End}}
\newcommand{\Ann}{\text{Ann}}
\newcommand{\Spec}{\text{Spec}}

\newcommand{\Bunantisym}{\text{Bun}^{\sigma, +}_{\text{SL}_{n}(C)}}
\newcommand{\HBunantisym}{\text{HBun}^{\sigma, +}_{\text{SL}_n(C)}}
\newcommand{\HBunantisymss}{\text{HBun}^{\sigma, + ; \text{ss}}_{\text{SL}_n(C)}}
\newcommand{\Bunantialt}{\text{Bun}^{\sigma, \tau }_{\text{SL}_n(C)}}
\newcommand{\HBunantialt}{\text{HBun}^{\sigma, \tau}_{\text{SL}_n(C)}}
\newcommand{\HBunantialtss}{\text{HBun}^{\sigma, \tau ; \text{ss}}_{\text{SL}_n(C)}}
\newcommand{\HBun}{\text{HBun}(C)}
\newcommand{\HBunrd}{\text{HBun}_{\text{r,d}}(C)}
\newcommand{\HBunrdss}{\text{HBun}_{\text{r,d}}^{ss}(C)}
\newcommand{\HBunsl}{\text{HBun}_{\text{SL}_n}(C)}
\newcommand{\HBunslss}{\text{HBun}^{\text{ss}}_{\text{SL}_n}(C)}
\newcommand{\gmssym}{M^{\sigma, +}_n}
\newcommand{\gmsalt}{M^{\sigma, \tau}_n}
\newcommand{\Bun}{\text{Bun}(C)}
\newcommand{\Bunsl}{\text{Bun}_{\text{SL}_n}(C)}
\newcommand{\HCoh}{\text{HCoh}(C)}
\newcommand{\Coh}{\text{Coh}(C)}
\newcommand{\wt}{\text{wt}_{\mathbb{G}_m}}
\newcommand{\Ldet}{\mathcal{L}_{\text{det}}}
\newcommand{\QCoh}{\text{QCoh}}

\begin{document}

\begin{abstract}
We study the moduli of anti-invariant Higgs bundles as introduced by Zelaci. Using recent existence results of Alper, Halpern-Leistner and Heinloth we establish the existence of a separated good moduli space for semistable anti-invariant Higgs bundles. Along the way this produces a non-GIT proof of the existence of a separated good moduli space for semistable Higgs bundles. We also prove the properness of the Hitchin system in this setting.
\end{abstract}

\maketitle

\section{Introduction}

\subsection{Motivation}

Higgs bundles were introduced by Hitchin in \cite{Hitchin1} \cite{Hitchin2}, where it was shown that there exists a moduli space of stable Higgs bundles, and that this admits a morphism, often called the Hitchin system, to an affine space. A Higgs bundle consists of a vector bundle $E$ on a curve $c$ and a section $\varphi \in H^0(C, \End(E) \times K_C)$. The Hitchin system can be interpreted as giving the coefficients of the characteristic polynomial of $\varphi$. It was furthermore proven that this morphism is projective (if the rank and degree are coprime), and with a natural symplectic structure on $M^s$, forms a completely algebraically integrable system. While the results of Hitchin could be interpreted in terms of algebraic geometry, this construction was gauge-theoretic in nature.

Nitsure \cite{nitsure1991moduli} and Simpson \cite{Simpson1} \cite{Simpson2} extended this futher, by constructing a moduli space of semistable Higgs bundles with algebro-geometric techniques, which ensures that the Hitchin fibration is always projective. Nitsure moreover extended this construction to the case where $\varphi \in H^0(C, \End(E) \otimes L)$ where $L$ is an arbitrary line bundle on $C$. 

This story can be generalised in many directions, one of which is to look at parabolic Higgs bundles as introduced by Simpson \cite{Simpson3}. Parabolic bundles were introduced by Mehta and Seshadri in \cite{MehtaSeshadri}. Given a reduced divisor $D$, a parabolic bundle consists of a vector bundle $E$ together with a flag in the fiber $E_p$ for $p \in D$. This can be extended to Higgs bundles by also considering a Higgs field $\varphi \in H^0(C, \End(E) \otimes L(D))$ that is compatible with the flags at points of $D$. This compatibility can be understood in two ways, and this leads to either a symplectic or a Poisson structure respectively. The first is a Higgs field nilpotent with respect to the flag, often called a strongly parabolic Higgs field, studied in \cite{Konno}. The second case is preserving the flag, often called the weakly parabolic Higgs field, and these are studied in \cite{Yokogawa} and \cite{LogaresMartens}. There is a natural principal bundle analog, where one considers a principal $G$-bundle for a reductive group $G$ together with reductions to parabolic subgroups of a specific type at $p \in D$, and the Higgs field is a section of the adjoint bundle (tensored with $L(D)$) satisfying some compatibility.

This perspective leads to a generalisation to parahoric group schemes. Parahoric groups are certain subgroups of $G(K)$ for $K$ the fraction field of a DVR $R$, that behave similarly to parabolic subgroups of reductive groups over algebraically closed fields. One example is the inverse image of a parabolic subgroup under the reduction morphism $G(R) \rightarrow G(k)$, where $k$ is the residue field of $R$. These were studied in \cite{BruhatTits1} and \cite{BruhatTits2} and in particular, it was shown that for any parahoric group $P$, there exists a group scheme $\mathcal{G}$ over $\Spec R$ such that $\mathcal{G}(R) = P$ and $\mathcal{G}|_K = G \times \Spec K$. 

Using this, one can construct group schemes $\mathcal{G}$ over the curve $C$ such that the generic fiber is $G$, but at points $p \in D$, the group scheme over a formal neighborhood of $p$ is of the type prescribe above. In the case where all parahoric subgroups are derived from parabolic ones, torsors for these group schemes correspond to parabolic bundles. These group schemes and their torsors were considered in \cite{pappas2010some} and \cite{heinloth2010uniformization}. There is distinction between split parahoric group schemes and non-split ones, for the former one has $\mathcal{G}|_{C \backslash D} \cong G \times (C \backslash D)$. These group schemes and their torsors were considered in \cite{pappas2010some} and \cite{heinloth2010uniformization}.

Given such group schemes, one can then consider again Higgs bundles  for them, given by a torsor $\mathcal{E}$ for these group schemes and a section $\varphi \in H^0(C, \text{ad}(\mathcal{E}) \otimes L(D))$. Studying these parahoric Higgs bundles form the main motivation for this work. These were studied before in the strong case in \cite{BaragliaKamgarpourVarma}, where it was shown that the Hitchin system here forms a completely integrable system. The weak case was considered in \cite{kydonakissunZhao} where a moduli space of semistable parahoric Higgs bundles was constructed using GIT and it was shown that this moduli space has a Poisson structure. Both these results deal only with the case of split parahoric group schemes, whereas the objects of interest in this paper correspond to a non-split parahoric group scheme. Our approach in this paper differs from the above results, as will be further expanded below, by using stack-theoretic results from \cite{alper2023existence} to prove existence and properness of these moduli spaces. Even in the case of ordinary Higgs bundles, this is new to the literature.

In \cite{balaji2010moduli}, it is shown that given a split parahoric group scheme $\mathcal{G}$ with generic fiber $G$, there exists a cover $\tilde{C} \rightarrow C$ with Galois group $\Gamma$ such that the ramification divisor is $D$, and an action of $\Gamma $ on $G$ such that $\mathcal{G}$-bundles correspond to $\Gamma$-equivariant $G$-bundles on $\tilde{C}$. This was extended to non-split schemes in \cite{DamioliniHong}. This allows to translate the study of parahoric Higgs bundles to simpler objects on the cover.

One example of this is anti-invariant bundles, which correspond to a $\Gamma = \Z / 2Z$-cover and the action on $\text{SL}_n$ corresponding to $g \mapsto ^{t}{g^{-1}}$. These objects were studied in detail by Zelaci in \cite{zelaci2019moduli} \cite{zelaci2017moduli} \cite{hacen2022hitchin} and in the arithmetic setting by Laumon and Ng\^{o} in \cite{LaumonNgo}. This notion also generalises the Prym variety from line bundles to higher rank vector bundles.

If $C$ is a curve with an involution $\sigma$, then an anti-invariant Higgs bundle is a triple $(E, \psi, \varphi)$ consisting of a vector bundle $E$ (of trivial determinant), an isomorphism $\psi : \sigma^* E \rightarrow E^*$ and a Higgs field $\varphi \in H^0(E, \End(E) \otimes L)$, where $L$ is an arbitrary anti-invariant line bundle on $C$, i.e. $\sigma^* L \cong L^*$. The isomorphism $\psi$ and the Higgs field $\varphi$ are also required to satisfy the compatibility condition of Diagram \ref{eq:antiinvHiggs}.

Another parallel motivation for this work, is as an application of the beyond GIT formalism developed in \cite{halpern2014structure} and \cite{alper2023existence}. In the latter, criteria are given for when a proper good moduli space of a stack exists, as introduced in \cite{alper2013good}. This reduces to the verification of two valuative criteria, which both look like a filling condition of a stack of the form $[\Spec S / \Gm]$ for over a codimension two point $O$ which is fixed for the $\Gm$-action. In the first case, we have $S = R[t]$ for $R$ a DVR and weight 1 on $t$. In the second case, $S = R[x,y]/ (xy - \pi)$, where $\pi$ is a uniformiser of $R$, and $x,y$ have weights 1 and -1 respectively.

\subsection{Main results}

Our first main result, Theorem \ref{th:gmsexists} shows the existence of a good moduli space for semistable anti-invariant Higgs bundles. This is done with the techniques of \cite{alper2023existence}. Our strategy consists of first showing that these criteria hold for ordinary Higgs bundles, and then showing that given this, we can also extend the isomorphism. Along the way, this also produces a stacky proof of the existence of the good moduli space of semistable Higgs bundles, without resorting to GIT, Theorem \ref{th:gmsH}, which is new to the literature as far as the author is aware. 

The second main result, Theorem \ref{th:Hitchinproper}, shows that the Hitchin system for this setting is a proper morphism. This is also done using results of \cite{alper2023existence} that allow us to reduce to showing that the morphism from the stack of anti-invariant Higgs bundles to the Hitchin base $B$ satisfies the existence part of the valuative criterion of properness. This is then shown by essentially combining the known corresponding result for the stacks of Higgs bundles and anti-invariant bundles.    

\subsection{Outline}

In Section 2 and 3 of this paper, we review the definitions of (stacks of) Higgs bundles and their anti-invariant counterpart, and establish some of the basic results about them. Section 4 is concerned with analysing filtrations of (anti-invariant) Higgs bundles, or maps from $[\A^1 / \Gm]$. This serves the dual purpose of motivating the definition of semistability for anti-invariant Higgs bundles and constructing a $\Theta$-stratification of the stack of anti-invariant Higgs bundles, which is later used to simplify the proof of properness of the Hitchin system. Section 5 summarizes the existence result of \cite{alper2023existence}, together with the valuative criteria introduced there. These are applied in Section 6 to the stack of semistable Higgs bundles and in Section 7 to the anti-invariant case. Finally, the properness of the Hitchin system is proven in Section 8.

\subsection{Notation}

Throughout the paper, $C$ will always denote a smooth projective curve of genus $g \geq 2$ that comes equipped with an involution $\sigma$, and $L$ is a line bundle on $C$ with $\deg(L) \geq 2g -2 $, which apart from in Section 2 and 6, is assumed to be anti-invariant, i.e. $\sigma^*L \cong L^*$.

\subsection{Acknowledgements}

This work has benefited from discussions with Jarod Alper and Jochen Heinloth. I also thank Johan Martens for introducing me to this topic, and advising me throughout. 

\section{Ordinary Higgs bundles}

\subsection{Definition + necessary properties}

 Given a smooth projective curve $C$, a \textbf{Higgs bundle} on $C$ is a vector bundle $E$ on $C$ together with an endomorphism $\varphi \colon E \rightarrow E \otimes L$. The endomorphism $\varphi$ is called the \textbf{Higgs field}. We  will mostly be interested in the case of Higgs bundles with trivial determinant or $\text{SL}_r$-Higgs bundles. These are given by a Higgs bundle $(E, \varphi)$ together with a trivialisation of the determinant $\delta \colon \Lambda^r E \rightarrow \mathcal{O}_C$ such that $\varphi$ is trace-free.

\begin{definition}
    For a scheme $S$, we define $\HBun(S)$ as the groupoid of coherent sheaves $\mathcal{E}$ on $C \times S$ such that for each $s \in S$ the fiber $\mathcal{E}_s$ is a vector bundle on $C$ together with a morphism $\varphi \colon \mathcal{E} \rightarrow \mathcal{E} \otimes \pi^*L$, where $\pi: C \times S \rightarrow C$ is the projection onto the first factor. This defines the \textbf{stack of Higgs bundles on $C$}. By fixing the rank $r$ and degree $d$ in this definition, we obtain a connected stack that we denote by $\HBunrd$.

    Considering the groupoid of Higgs bundles together with a trivialisation of the determinant of $\mathcal{E}$ and a trace free Higgs field defines the \textbf{stack of Higgs bundles with trivial determinant} $\HBunsl$.
\end{definition}

These are all algebraic stacks, locally of finite type, see for example \cite[Theorem 7.18]{casalaina2018introduction}. 

\begin{proposition}
\label{prop:bunaffdiag}
    The stacks $\HBun \text{ and } \HBunsl$ have affine diagonal.
\end{proposition}
\begin{proof}
    We need to show that for a morphism $(a,b) \colon S \rightarrow \HBun\times \HBun$ the algebraic space $\underline{\text{Isom}}_{\HBun(S)}(a,b)$ is a scheme affine over $S$. The morphisms $a,b$ correspond to Higgs bundles $(E_a, \varphi_a), (E_b, \varphi_b)$ on $C \times S$ and elements of $\underline{\text{Isom}}_{\HBun(S)}(a,b)$ correspond to vector bundle isomorphisms $\psi \colon E_a \rightarrow E_b$ that commute with the Higgs fields, i.e. the following diagram is commutative 

    % https://q.uiver.app/#q=WzAsNCxbMCwwLCJFX2EiXSxbMiwwLCJFX2EgXFxvdGltZXMgXFxwaV4qIEwiXSxbMCwyLCJFX2IiXSxbMiwyLCJFX2IgXFxvdGltZXMgXFxwaV4qIEwiXSxbMCwxLCJcXHZhcnBoaV9hIl0sWzIsMywiXFx2YXJwaGlfYiJdLFswLDIsIlxccHNpIiwyXSxbMSwzLCJcXHBzaSBcXG90aW1lcyBcXHRleHR7aWR9X3tcXHBpXipMfSJdXQ==
\begin{equation} \label{eq:isohiggs}
    \begin{tikzcd}
	{E_a} && {E_a \otimes \pi^* L} \\
	\\
	{E_b} && {E_b \otimes \pi^* L.}
	\arrow["{\varphi_a}", from=1-1, to=1-3]
	\arrow["{\varphi_b}", from=3-1, to=3-3]
	\arrow["\psi"', from=1-1, to=3-1]
	\arrow["{\psi \otimes \text{id}_{\pi^*L}}", from=1-3, to=3-3]
\end{tikzcd}
\end{equation}
        
    There  is a natural forgetful morphism $\HBun \rightarrow \Bun$ from the stack of Higgs bundles on $C$ to the stack of vector bundles on $C$ forgetting the Higgs field. This gives rise to a morphism $(a',b'): S \rightarrow \Bun \times \Bun$, which corresponds to the vector bundles $E_a,E_b$. Elements of $\underline{\text{Isom}}_{\Bun(S)}(a',b')$ are vector bundle isomorphisms $f: E_a \rightarrow E_b$. 
    
    Because $\Bun$ has affine diagonal, we know that $\underline{\text{Isom}}_{\Bun(S)}(a',b')$ is a scheme affine over $S$, and it is easy to see from the above interpretation of its elements that $\underline{\text{Isom}}_{\HBun(S)}(a,b)$ is a subscheme. The condition of commuting in Diagram \ref{eq:isohiggs} with the Higgs fields is a closed one, so it follows that $\underline{\text{Isom}}_{\HBun(S)}(a,b)$ is an affine scheme over $S$ and $\HBun$ has affine diagonal.

    For $\HBunsl$ the result follows using a similar description of the isomorphisms as vector bundle isomorphisms that commute with both the Higgs fields and the respective trivialisations of the determinant.
\end{proof}

We will make use of the following enlargement of the stack of Higgs bundles.

\begin{definition}
For a scheme $S$, we define $\HCoh(S)$ as the groupoid of coherent sheaves $\mathcal{E}$ on $C \times S$ together with a morphism $\varphi \colon\mathcal{E} \rightarrow \mathcal{E} \otimes \pi^*L$. This defines the \textbf{stack of Higgs sheaves on $C$}, denoted by \HCoh.
\end{definition}

The stack $\HCoh$ is algebraic, locally of finite type. There is a natural forgetful morphism $\HCoh \rightarrow \Coh$ to the stack of coherent sheaves on $C$ forgetting the Higgs field. Because $\HBun$ is the preimage of $\Bun$ under this morphism, it follows that $\HCoh$ contains $\HBun$ as an open substack.

\begin{proposition}\label{prop:Hcohaff}
    $\HCoh$ is affine over $\Coh$.
\end{proposition}
\begin{proof}
    If $S \rightarrow \Coh$ is a morphism from a scheme $S$ corresponding to a family of coherent sheaves $\mathcal{E}$, the fiber product of this morphism with $\HCoh$ is given by $H^0(\End(\mathcal{E}) \otimes \pi^* L)$. It follows that the morphism $\HCoh \rightarrow \Coh$ is affine.
\end{proof}

\section{Anti-invariant Higgs Bundles}

From now on we assume $C$ is equipped with an involution $\sigma$. Anti-invariant vector bundles on such a curve were studied by Zelaci in \cite{zelaci2019moduli}, \cite{zelaci2017moduli} and \cite{hacen2022hitchin}. 

\begin{definition}
An \textbf{anti-invariant vector bundle} $(E, \psi, \delta)$ is a vector bundle $E$ on $C$ together with a trivialisation of the determinant $\delta$ and an isomorphism
\begin{equation*}
    \psi \colon \sigma^* E \rightarrow E^*.
\end{equation*}
The isomorphism $\psi$ is said to be
\begin{enumerate}
    \item \textbf{$\sigma$-symmetric} if $^{t}{\sigma^* \psi} = \psi.$
    \item  \textbf{$\sigma$-alternating} if $^{t}{\sigma^* \psi} = -\psi.$ The \textbf{type} $\tau$ of a $\sigma$-alternating vector bundle is defined as 
    \begin{equation*}
        \tau = (\tau_p)_{p \in R} \mod \pm 1
    \end{equation*}
    where $R$ is the ramification divisor of $\sigma$ and $\tau_p \in \lbrace \pm 1 \rbrace$ is the Pfaffian of the isomorphism $\psi \colon (\sigma^* E)_p \rightarrow (E^*)_p$.
\end{enumerate}
\end{definition} 

\begin{remark}\ 
\begin{enumerate}
    \item As outlined in \cite[Section 3.2.1]{hacen2022hitchin}, if $E$ is a stable vector bundle, then the isomorphism $\psi$ is unique up to scalar multiplication. This implies that $^{t}{\sigma^*\psi} = \pm \psi$, so such bundles are always either $\sigma$-symmetric or $\sigma$-alternating.
    \item We are mainly interested in the $\sigma$-symmetric and $\sigma$-alternating cases, but the results discussed throughout the rest of this paper hold without any restrictions on the isomorphism $\psi$.
    \item From now on, we will not always write out the trivialisation $\delta$, but it will always be implicitly assumed.
\end{enumerate}
\end{remark}

\begin{definition}
    We define the \textbf{stack of $\sigma$-symmetric anti-invariant bundles} $\Bunantisym$ (resp. \textbf{$\sigma$-alternating anti-invariant bundles} $\Bunantialt$) by assigning to  a scheme $S$ the groupoid consisting of a coherent sheaf $\mathcal{E}$ on $C \times S$ such that for each $s \in S$ the sheaf $\mathcal{E}_s$ is a vector bundle, a trivialisation $\delta$ of the determinant of $\mathcal{E}$ and a $\sigma$-symmetric (resp. $\sigma$-alternating of type $\tau$) isomorphism $\psi: \sigma^* \mathcal{E} \rightarrow \mathcal{E}^*$.
\end{definition}

Let $\mathcal{G}$ be a group scheme on $C$ with a lift of the involution $\sigma$. A $(\mathcal{\sigma}, \mathcal{G})$-bundle as considered in \cite[Definition 2.2.3]{balaji2010moduli} is given by a $\mathcal{G}$-torsor $E$ on $C$ and a lift of the action of $\sigma$ to $E$ such that for $p \in E$ and $g \in \mathcal{G}$ we have $\sigma(g \cdot p) = \sigma(g) \cdot \sigma(p).$ It is shown in \cite[Proposition 2.4]{zelaci2017moduli} that anti-invariant vector bundles on $C$ correspond to $(\sigma, \mathcal{G})$-bundles for various group schemes $\mathcal{G}.$ For $\sigma$-symmetric bundles the associated group scheme is the constant group scheme $\text{SL}_r \times C$ with action $\sigma(g, x) = (^{t}{g^{-1}})$. For $\sigma$-alternating vector bundles, there is a very analogous construction with a slightly more complicated group scheme (that is still \'etale locally isomorphic to $\text{SL}_r$). 

Furthermore, in \cite[Theorem 4.1.5]{balaji2010moduli} it is shown that $(\sigma, \mathcal{G})$-bundles correspond to torsors for the group scheme $\mathcal{H} = \text{Res}_{C / C'}(\mathcal{G})^\sigma$ of $\sigma$-invariants in the Weil restriction of the group scheme to the quotient $C' = C / \sigma$. In the case of anti-invariant bundles, it is also shown in \cite[Proposition 2.4]{zelaci2017moduli} that this Weil restriction is a parahoric group scheme as defined in \cite{pappas2010some}. It subsequently follows from \cite[Proposition 2.1]{heinloth2010uniformization} that $\Bunantisym$ and $\Bunantialt$ are smooth algebraic stacks, locally of finite type. 

\begin{remark}
    Anti-invariant vector bundles exhibit some analogies with symplectic and orthogonal vector bundles. If $p$ is a fixed point of $\sigma$, it is shown that the fiber of the parahoric group schemes over $p$ is an extension of $\text{SO}_r$ in the $\sigma$-symmetric case, and $\text{Sp}_r$ for $\sigma$-alternating.  
\end{remark}

From now on, we also assume that the line bundle is anti-invariant with a fixed isomorphism, i.e. $\sigma^*L \cong L$. This assumption can be made if $L$ is the canonical bundle, for example. 

\begin{definition}
    An \textbf{anti-invariant Higgs bundle}  $(E, \psi, \delta, \varphi)$ on $C$ is an anti-invariant vector bundle $(E, \psi)$ together with a  traceless vector bundle morphism $\varphi \colon E \rightarrow E \otimes L$ such that 
    \begin{equation}\label{eq:antiinvHiggs}
        % https://q.uiver.app/#q=WzAsNCxbMCwwLCJcXHNpZ21hXipFIl0sWzIsMCwiRV4qIl0sWzAsMiwiXFxzaWdtYV4qRSBcXG90aW1lcyBMIl0sWzIsMiwiRV4qXFxvdGltZXMgTCJdLFsxLDMsIi0ge157dH1cXHZhcnBoaX0iXSxbMCwxLCJcXHBzaSJdLFsyLDMsIlxccHNpIFxcb3RpbWVzIGlkX0wiLDJdLFswLDIsIlxcc2lnbWFeKiBcXHZhcnBoaSIsMl1d
\begin{tikzcd}
	{\sigma^*E} && {E^*} \\
	\\
	{\sigma^*E \otimes L} && {E^*\otimes L}
	\arrow["{- {^{t}\varphi}}", from=1-3, to=3-3]
	\arrow["\psi", from=1-1, to=1-3]
	\arrow["{\psi \otimes id_L}"', from=3-1, to=3-3]
	\arrow["{\sigma^* \varphi}"', from=1-1, to=3-1]
\end{tikzcd}
    \end{equation}
    We call an anti-invariant Higgs bundle \textbf{$\sigma$-symmetric} (resp. \textbf{$\sigma$-alternating} ) when the underlying anti-invariant bundle is.
\end{definition}

\begin{remark}
\begin{enumerate}
    \item If we consider $\sigma^*E$ and $E^*$ as Higgs bundles with respective Higgs fields $\sigma^* \varphi$ and $-{^{t}\varphi}$, then this definition is the statement that $\psi$ is an isomorphism of Higgs bundles. The reason for the minus sign is that the involution on the group scheme acts infinitesimally by a minus sign on the Lie algebra. This ensures that anti-invariant Higgs bundles correspond to the Higgs analog of $(\sigma, \mathcal{G})$-bundles and hence also parahoric Higgs bundles.
    \item In the case of parabolic Higgs bundles, there is a distinction between strong and weak parabolic Higgs bundles. The former are nilpotent with respect to the flag at distinguished points, whereas the latter only preserves the flag. This distinction is still important in the parahoric case and could also be made here. It would correspond to either requiring the Higgs field to vanish on the ramification divisor of $C$ in the strong case, or just sticking with our definition above.
\end{enumerate}
\end{remark}

\begin{definition}
    We define the \textbf{stack of $\sigma$-symmetric anti-invariant Higgs bundles} $\HBunantisym$ (resp. \textbf{$\sigma$-alternating anti-invariant bundles} $\HBunantialt$) by assigning to  a scheme $S$ the groupoid consisting of a coherent sheaf $\mathcal{E}$ on $C \times S$ such that for each $s \in S$ $\mathcal{E}_s$ is a vector bundle, a trivialisation $\delta$ of the determinant of $\mathcal{E}$, a $\sigma$-symmetric (resp. $\sigma$-alternating of type $\tau$) isomorphism $\psi \colon \sigma^* \mathcal{E} \rightarrow \mathcal{E}^*$, and an endomorphism $\varphi \colon \mathcal{E} \rightarrow \mathcal{E} \otimes \pi^* L$ such that the analog of Diagram \ref{eq:antiinvHiggs} commutes.
\end{definition}

\begin{remark}
    If we assume that the underlying bundle is stable, it is shown in \cite[Theorem 3.12]{hacen2022hitchin} that an anti-invariant Higgs bundle corresponds to a cotangent vector at the underlying anti-invariant bundle. Similarly, if we allow derived stacks, the above stacks are the truncation of the cotangent stacks to $\Bunantisym$ and $\Bunantialt$.  This shows that these are algebraic stacks locally of finite type. 
\end{remark}

\begin{proposition}
    The stacks $\HBunantisym$ and $\HBunantialt$ have affine diagonal.
\end{proposition}
\begin{proof}
    This is proven analogously to Proposition \ref{prop:bunaffdiag}. 
\end{proof}

 The stack of anti-invariant Higgs bundles admits a number of useful forgetful morphisms, collected in the following commutative diagram. 

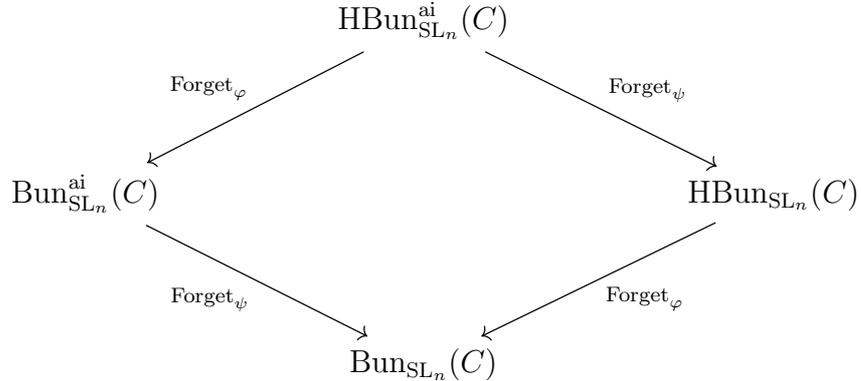
\begin{equation*}
    % https://q.uiver.app/#q=WzAsNCxbMiwwLCJcXHRleHR7SEJ1bn1ee1xcdGV4dHthaX19X3tcXHRleHR7U0x9X259KEMpIl0sWzAsMiwiXFx0ZXh0e0J1bn1ee1xcdGV4dHthaX19X3tcXHRleHR7U0x9X259KEMpIl0sWzQsMiwiXFx0ZXh0e0hCdW59X3tcXHRleHR7U0x9X259KEMpIl0sWzIsNCwiXFx0ZXh0e0J1bn1fe1xcdGV4dHtTTH1fbn0oQykiXSxbMCwxLCJcXHRleHR7Rm9yZ2V0fV9cXHZhcnBoaSIsMl0sWzAsMiwiXFx0ZXh0e0ZvcmdldH1fXFxwc2kiXSxbMSwzLCJcXHRleHR7Rm9yZ2V0fV9cXHBzaSIsMl0sWzIsMywiXFx0ZXh0e0ZvcmdldH1fXFx2YXJwaGkiXV0=
\begin{tikzcd}
	&& {\text{HBun}^{\text{ai}}_{\text{SL}_n}(C)} \\
	\\
	{\text{Bun}^{\text{ai}}_{\text{SL}_n}(C)} &&&& {\text{HBun}_{\text{SL}_n}(C)} \\
	\\
	&& {\text{Bun}_{\text{SL}_n}(C)}
	\arrow["{\text{Forget}_\varphi}"', from=1-3, to=3-1]
	\arrow["{\text{Forget}_\psi}", from=1-3, to=3-5]
	\arrow["{\text{Forget}_\psi}"', from=3-1, to=5-3]
	\arrow["{\text{Forget}_\varphi}", from=3-5, to=5-3]
\end{tikzcd}
\end{equation*}

In this diagram, the arrows pointing to the left are affine.

\section{Filtrations and Semistability}

\subsection{Filtrations}

Let $\Theta = [\A^1 / \Gm]$. This stack has two geometric points which we will denote by $0$ and $1$. For the purposes of defining the stack of semistable anti-invariant Higgs bundles it will be important to understand what morphisms from $\Theta$ to $\HBunantisym$/$\HBunantialt$ look like. Translating this maps into $\Gm$-equivariant information on $\A^1 \times C$ is the purpose of the following lemma.

\begin{lemma}\label{lem:Gmequivhbunai}
A morphism from $\Theta$ to $\HBunantisym$ (resp. $\HBunantialt$) is equivalent to giving a $\Gm$-equivariant vector bundle $\mathcal{E}$ over $\AC$ together with a $\Gm$-equivariant $\sigma$-symmetric (resp. $\sigma$-alternating of type $\tau$) isomorphism $\psi \colon \sigma^* \mathcal{E} \rightarrow \mathcal{E}^*$\footnote{Here $\sigma$ denotes the induced involution on $\AC$, and $\sigma^* \mathcal{E}$ and $\mathcal{E}^*$ are endowed with the pullback and the dual of the action on $\mathcal{E}$ respectively.} and a $\Gm$-equivariant morphism $\varphi: \mathcal{E} \rightarrow \mathcal{E} \otimes K$ compatible with $\psi$.
\end{lemma}
\begin{proof}
   From the quotient morphism $\A^1 \rightarrow \Theta$, we get a family of anti-invariant Higgs bundles on $C$ over $\A^1$. Factoring through the quotient stack means precisely that the bundle and corresponding sections are $\Gm$-equivariant. 
\end{proof}

It is well-known that $\Gm$-equivariant vector bundles on $\AC$ correspond to filtered vector bundles on $C$. We will show that the two equivariant maps serve to put restrictions on the corresponding filtration of an anti-invariant Higgs bundle on $C$. For this purpose we recall one way of constructing this equivalence, following \cite{asok2006equivariant}. 

Given a vector bundle $E$ with an increasing filtration $(F^p(E))_{p \in \Z}$, we can produce a graded $\mathcal{O}_C[t]$-module given by 
\begin{equation}
    R(E,F^{\bullet}) = \bigoplus\limits_{p \in \Z} x^p F^p(E).
\end{equation}
This is known as the Rees construction. Because $\Gm$-equivariant morphisms of $\Gm$-equivariant vector bundles have to respect the grading, it follows that $\Gm$-equivariant morphisms on $\A^1$ correspond to endomorphisms between the corresponding vector spaces that respect the filtrations. For $\varphi$, this implies that we have to consider filtrations by Higgs subbundles, i.e subbundles $F$ such that $\varphi(F) \subset F \otimes L$. To deduce a similar conclusion for $\psi$, we have to figure out what the filtration $\mathcal{E}^*$ induced by the dual action looks like given a filtration of $E$. This is standard, but we could not find a reference.

In \cite[Section 3]{asok2006equivariant} the equivalence between filtered vector spaces and $\Gm$-equivariant bundles on $\A^1$ is outlined. To recover the  filtration given a $\Gm$-equivariant bundle $\mathcal{E}$ on $\A^1$, for $U = \A^1 \backslash 0$ define the space
\begin{equation}
    E = \Gamma(U, \mathcal{V}|_U)^\Gm
\end{equation}
of $\Gm$-invariant sections away from the origin. This recovers the underlying vector space and is isomorphic to the fiber away from 0. We can define an increasing filtration on this vector space through
\begin{equation}
    F^p(E) = \lbrace s \in E \, | \, x^p s \text{ extends to a global section.} \rbrace
\end{equation}
where $x$ is a coordinate on $\A^1$. A useful fact is that $\Gamma(\A^1, \mathcal{E})$ is generated by the set of all $x^p s$ for $p \in \Z$ and $s \in F^p(E)$. 

To recover the filtration on the dual vector bundle, the following lemma will be helpful.
\begin{lemma}\label{lem:extenddual}
    A section $ \alpha \in \Gamma(U, \mathcal{E}^*)$ extends to $\A^1$ if and only if for every global section $s \in \Gamma(\A^1, \mathcal{E}) \subset \Gamma(U, \mathcal{E}\mid_U)$, the function $\alpha(s) \in \Gamma(U)$ extends to a global function.
\end{lemma}
\begin{proof}
    We will prove a slightly more general result. Suppose that $M$ is an $R[t]$-module for some commutative ring $R$. The setup of the lemma corresponds to a morphism $ \alpha \colon N = M \otimes_{R[t]} R[t, t^{-1}] \rightarrow R[t, t^{-1}]$ in the case $R = \C$. We want to show that there exists a morphism $\beta \colon M \rightarrow R[t]$ such that $\alpha = \beta \otimes 1_{R[t, t^{-1}]}.$ Furthermore, we can interpret $M$ as a submodule of $N$. We are moreover given that $\alpha(M) \subset R[t]$. This corresponds to the following commutative diagram
    % https://q.uiver.app/#q=WzAsNCxbMCwwLCJOIl0sWzAsMiwiTSJdLFsyLDIsIlJbdF0iXSxbMiwwLCJSW3QsdF57LTF9XSJdLFswLDMsIlxcYWxwaGEiXSxbMSwyLCJcXGFscGhhfF9NIl0sWzEsMCwiIiwxLHsic3R5bGUiOnsidGFpbCI6eyJuYW1lIjoiaG9vayIsInNpZGUiOiJ0b3AifX19XSxbMiwzLCIiLDEseyJzdHlsZSI6eyJ0YWlsIjp7Im5hbWUiOiJob29rIiwic2lkZSI6InRvcCJ9fX1dXQ==
\[\begin{tikzcd}
	N && {R[t,t^{-1}]} \\
	\\
	M && {R[t]}
	\arrow["\alpha", from=1-1, to=1-3]
	\arrow["{\alpha|_M}", from=3-1, to=3-3]
	\arrow[hook, from=3-1, to=1-1]
	\arrow[hook, from=3-3, to=1-3]
\end{tikzcd}\]
It thus follows that $\alpha|_M$ is the desired extension.
\end{proof}
Given this lemma, the filtration of the dual bundle is characterised by the following proposition.
\begin{proposition}\label{prop:dualfiltration}
    Let $\mathcal{E}$ be a $\Gm$-equivariant bundle on $\A^1$ corresponding to a filtered vector space $(E, F^p(E))$. Then the $\Gm$-equivariant vector bundle $\mathcal{E}^*$ corresponds to the filtered vector space $(E^*, \Ann(F^{-p + 1}(E)))$.
\end{proposition}
\begin{proof}
    Let $\alpha$ be a $\Gm$-invariant section of $\mathcal{E}^*$ over $U$. Using Lemma \ref{lem:extenddual}, we see that for $p \in \Z$ $x^p(\alpha)$ extends to a global section  if and only if for all $k \in Z$ and $v \in F^k(E)$ 
    \begin{equation*}
        x^p \alpha(x^k v) = x^{p + k} \alpha(v)
    \end{equation*} 
    extends to a global function. Because $\alpha$ and $v$ are both $\Gm$-invariant, it follows that $\alpha(v)$ is a $\Gm$-invariant function, hence a constant function. When $\alpha(v) \neq 0$ this extends if and only if 
    \begin{equation*}
        - k \leq p.
    \end{equation*}
    It follows that $x^p \alpha$ extends if and only if for all $k \in \Z$ such that the $-k > p$, we have that 
    \begin{equation*}
        \alpha|_{F^k(V)} = 0.
    \end{equation*}
    because the filtration on $V$ is increasing, this is equivalent to 
    \begin{equation*}
        \alpha|_{F^{-p + 1}(V)} = 0.
    \end{equation*}
    This proves that 
    \begin{equation*}
        F^p(V^*) = \Ann(F^{-p + 1}(V)). 
    \end{equation*}
\end{proof}

To apply this construction to $\Gm$-equivariant vector bundles on $\A^1 \times C$, we define a sheaf $E$ on $C$ through
\begin{equation*}
    E(V) = \Gamma(U \times V, \mathcal{E})^\Gm,
\end{equation*}
for an open $V \subset C$. Because a $\Gm$-invariant section is completely determined by its restriction to the fiber $\lbrace 1 \rbrace \times V$, this sheaf is isomorphic to $\mathcal{E}|_{\lbrace 1 \rbrace \times C}$. This gives us a vector bundle $E$ on $C$. 

We similarly define locally free subsheaves $F^p(E)$ of $E$ through
\begin{equation}
    F^p(E)(U \times V) = \lbrace s \in E(V) \, | \, x^p s \text{ extends to } \A^1 \times V \rbrace.
\end{equation}
The following lemma ensures that these are vector subbundles. 
\begin{lemma}
    The quotient $E/F^p(E)$ is torsion-free for all $p \in \Z$.
\end{lemma}
\begin{proof}
    Let $\alpha \in E/F^p(E)(V)$ for some open $V$ and suppose that there is a function $f$ on $C$ such that $f \alpha = 0$. If we pick a preimage $\tilde{\alpha} \in E(V)$, then this implies that $x^p f \tilde{\alpha} \in \mathcal{E}(U \times V) $ extends to a section on $\A^1 \times V$, or stated otherwise, the order of vanishing in $x$ of $x^p f \tilde{\alpha}$ is non-negative.  But since $f$ is a function on $C$, multiplying by $f$ cannot change the order of vanishing, and $x^p \tilde{\alpha}$ must also extend to a global section. This implies that $\alpha = 0 \in E/F^p(E)(V)$, and thus the quotient is torsion-free.
\end{proof}
The statement that
\begin{equation}
    \Gamma(\A^1 \times V, \mathcal{E}) = \langle x^p s \, | \, p \in \Z, s \in F^p(E) \rangle
\end{equation}
likewise remains true. 

The proof of Lemma \ref{lem:extenddual} also applies to a similar extension result for $\alpha \in \Gamma(U \times V, \mathcal{E}^*)$ to $A^1 \times V$. 

Putting all of this together, the proof of Proposition \ref{prop:dualfiltration} goes through, which gives the following result.

\begin{proposition}\label{prop:dualfiltrationbundle}
    Let $\mathcal{E}$ be a $\Gm$-equivariant bundle on $\A^1 \times C$ corresponding to a filtered vector bundle $(E, F^p(E))$ on $C$. Then the $\Gm$-equivariant vector bundle $\mathcal{E}^*$ corresponds to the filtered vector bundle $(E^*, \Ann(F^{-p + 1}(E)))$.
\end{proposition}

To see how this analysis imposes conditions on the filtrations, we define an orthogonal complement with respect to $\psi$. This notion is due to \cite[Proof of Proposition 4.2]{zelaci2019moduli}.

\begin{definition}
    Let $F$ be a subbundle of an anti-invariant bundle $(E, \psi)$. We define $F^{\perp_\sigma}$ as the kernel of 
    \begin{equation*}
        E \cong \sigma^* E^* \twoheadrightarrow \sigma^* F^*,
    \end{equation*}
    where the isomorphism is given by $\sigma^*\psi$.
\end{definition}
The isomorphism $\psi$ thus identifies $\sigma^* F^{\perp_\sigma}$ with $\Ann(F)$. Putting this together with Lemma \ref{lem:Gmequivhbunai} and Proposition \ref{prop:dualfiltration}, we obtain the following proposition.

\begin{proposition} \label{prop:filtonanti-invHiggs}
    Morphisms from $\Theta$ into $\HBunantisym$ (resp. $\HBunantialt$) correspond to a $\sigma$-symmetric (resp. $\sigma$-alternating of type $\tau$) anti-invariant Higgs bundle $E$ on $C$ together with a filtration by Higgs subbundles of the form 
    \begin{equation}
        0 = F^0 \subset F^1 \subset \ldots \subset F^k \subseteq {F^k}^{\perp_\sigma} \subset \ldots {F^1}^{\perp_\sigma} \subset {F^0}^{\perp_\sigma} = E.
    \end{equation}
\end{proposition}
\begin{proof}
    We have already seen that this corresponds to a filtration of $E$ by Higgs subbundles $(F^p_{p \in \Z}$ together with an isomorphism $\psi$ that respects the filtrations induced on $\sigma^* E$ and $E^*$.  This identifies $\sigma^*(F^p)$ with $\Ann(F^{-p + 1})$. By definition, this implies that 
    \begin{equation}
        F^p = {F^{-p + 1}}^{\perp_\sigma},
    \end{equation}
    and this allows us to rewrite the filtration in the desired form.
\end{proof}
\begin{remark}
    Note that the isomorphism $\psi$ induces an isomorphism of $\sigma^*({F^k}^{\perp_\sigma} / F^k)$ with $({F^k}^{\perp_\sigma} / F^k)^*$ and $\sigma^*(F^j/F^{j-1})$ with $({F^{j-1}}^{\perp_\sigma}/{F^j}^{\perp_\sigma})^*$. This ensures that the associated graded of the above filtration is again an anti-invariant Higgs bundle. This should be the case as the image of $0 \in \Theta$ corresponds to the associated graded of the filtration.   
\end{remark}

\subsection{Semistability}

To define an open substack of semistable anti-invariant Higgs bundles, we follow the methods of \cite{heinloth2018hilbert}, \cite{alper2023existence} and \cite{halpern2014structure}.

If $\mathcal{L}$ is a $\Gm$-equivariant line bundle on $\A^1$, then its global sections are given by $H^0(\A^1, \mathcal{L}) = k[x] \cdot e$, where $e$ is a section unique up to a non-zero scalar and $\Gm$ acts on it by $t \cdot e = t^d e$. We call the integer $d$ the \textbf{weight} of $\mathcal{L}$ and denote it by $\wt(\mathcal{L})$.

\begin{definition}\cite[Definition 1.2 and Remark 1.3]{heinloth2018hilbert}
    Let $\mathcal{X}$ be an algebraic stack and $\mathcal{L}$ a line bundle on $\mathcal{X}$. A point $x \in \mathcal{X}(K)$ is \textbf{semistable} with respect to $\mathcal{L}$ if for all morphisms $f \colon 
    \Theta_K \rightarrow \mathcal{X}$ with $f(1) \cong x$ and $f(0)  x$, 
    \begin{equation*}
        \wt(f^* \mathcal{L}) \leq 0
    \end{equation*}
    holds.
\end{definition}

The stack $\Bun$ has a natural line bundle $\Ldet$ whose fiber over a vector bundle $E$ is the determinant of its cohomology
\begin{equation*}
    \mathcal{L}_{\text{det}, E} = \text{det}(H^*(C, \End(E)))^{-1}.
\end{equation*}
A morphism $ f \colon \Theta \rightarrow \Bun$ corresponds to a filtered vector bundle $(E, (E^i)_{i \in \Z})$. The weight $\wt(f^* \Ldet)$ can be calculated in terms of the ranks and degrees of the subbundles $E^i$. In \cite[Section 1.E.c]{heinloth2018hilbert}  it is shown that 
\begin{equation}\label{eq:weightBun}
    \wt(f^* \Ldet) = 2 \sum\limits_{i \in \Z} \text{rk}(E^i)(\text{rk}(E)- \text{rk}(E^i))(\mu(E^i) - \mu(E / E^i)) , 
\end{equation}
where $\mu(E^i) = \frac{\text{deg}(E^i)}{\text{rk}(E^i)}$ is the slope of $E^i$. Analyzing this formula shows that this number is non-positive if and only if it is so for all 1-step filtrations, i.e. a single non-trivial subbundle. This gives the usual criterion for semistability, namely a vector bundle $E$ is semistable if and only if $\mu(F) \leq \mu(E)$ for all non-trivial subbundles $F$. In the case where $\text{deg}(E) = 0$, this reduces to the condition that $\text{deg}(F) \leq 0$.

By pulling back along the forgetful morphism $g: \HBun \rightarrow \Bun$, we get a line bundle $g^* \Ldet$ on the stack of Higgs bundles. A morphism $f: \Theta \rightarrow \HBun$ corresponds to a filtration of the bundle by Higgs subbundles. Composing with the forgetful morphism, we get a morphism $h: \Theta \rightarrow \Bun$. The weight can be computed as 
\begin{equation}
    \wt(f^* g^* \Ldet) = \wt(h^* \Ldet).
\end{equation}

We have
% https://q.uiver.app/#q=WzAsNixbMCwxLCJcXFRoZXRhIl0sWzIsMSwiXFx0ZXh0e0hCdW59KEMpIl0sWzQsMywiXFx0ZXh0e0J1bn0oQykiXSxbMiwwLCJnXiogXFxtYXRoY2Fse0x9X3tcXGRldH0iXSxbNCwyLCJcXG1hdGhjYWx7TH1fe1xcZGV0fSJdLFswLDAsImZeKmdeKiBcXG1hdGhjYWx7TH1fe1xcZGV0fSBcXGNvbmcgaF4qIFxcbWF0aGNhbHtMfV97XFxkZXR9Il0sWzAsMSwiZiJdLFszLDEsIiIsMix7InN0eWxlIjp7ImhlYWQiOnsibmFtZSI6Im5vbmUifX19XSxbNCwyLCIiLDIseyJzdHlsZSI6eyJoZWFkIjp7Im5hbWUiOiJub25lIn19fV0sWzEsMiwiZyIsMl0sWzAsMiwiaCIsMl0sWzUsMCwiIiwyLHsic3R5bGUiOnsiaGVhZCI6eyJuYW1lIjoibm9uZSJ9fX1dXQ==
\[\begin{tikzcd}
	{f^*g^* \mathcal{L}_{\det} \cong h^* \mathcal{L}_{\det}} && {g^* \mathcal{L}_{\det}} \\
	\Theta && {\text{HBun}(C)} \\
	&&&& {\mathcal{L}_{\det}} \\
	&&&& {\text{Bun}(C)}
	\arrow["f", from=2-1, to=2-3]
	\arrow[no head, from=1-3, to=2-3]
	\arrow[no head, from=3-5, to=4-5]
	\arrow["g"', from=2-3, to=4-5]
	\arrow["h"', from=2-1, to=4-5]
	\arrow[no head, from=1-1, to=2-1]
\end{tikzcd}\]

So, for semistability of Higgs bundles we get the same slope condition as for vector bundles, but only for filtrations by Higgs subbundles. This recovers the usual definition that a Higgs bundle $(E, \varphi)$ is semistable if and only if for every Higgs subbundle $F$, we have that $\mu(F) \leq \mu(E)$.

The line bundle $\Ldet$ extends to the stack $\HCoh$, but now to test for semistability, we also allow filtrations by Higgs subsheaves. For any Higgs sheaf, the maximal torsion subsheaf is always a Higgs subsheaf and its slope is always $+ \infty$. This implies that a semistable Higgs sheaf is torsion-free and hence a Higgs bundle.  

Similarly, for anti-invariant Higgs bundles, we get a slope condition on the filtrations described by Proposition \ref{prop:filtonanti-invHiggs}. Suppose $F$ is a subbundle of degree $d$ of an anti-invariant Higgs bundle $E$. Because $F^{\perp_\sigma}$ is isomorphic to $\Ann(F)$, it follows that the degree of $F^{\perp_\sigma}$ is also $d$, and $\text{rk}(F) = \text{rk}(E) - \text{rk}(F)$. Because of this, if 

\begin{equation*}
        0 = F^0 \subset F^1 \subset \ldots \subset F^k \subseteq {F^k}^{\perp_\sigma} \subset \ldots {F^1}^{\perp_\sigma} \subset {F^0}^{\perp_\sigma} = E.
\end{equation*}
is a filtration of $E$, then Equation \ref{eq:weightBun} reduces to 
\begin{equation}
    \wt(f^* \Ldet) = 4 \sum\limits_{i = 1}^k \text{rk}(E^i)(\text{rk}(E)- \text{rk}(E^i))(\mu(E^i) - \mu(E / E^i)).
\end{equation}

From this formula, we can again reduce to the case of 1-step filtrations. In this case, a 1-step filtration looks like 
\begin{equation*}
    0 \subset F \subset F^{\perp_\sigma} \subset E.
\end{equation*}
combining this with the fact that $\text{deg}(E) = 0$ leads to the following definition from \cite[Definition 4.1]{zelaci2019moduli}.
\begin{definition}
    Let $(E, \varphi, \psi)$ be an anti-invariant Higgs bundle. A subbundle $F$ is \textbf{$\sigma$-isotropic} if $F \subset F^{\perp_\sigma}$ \footnote{In \cite{zelaci2019moduli}, this is defined as the induced morphism $\psi: \sigma^* F \rightarrow F^*$ being zero, but this is equivalent to the one given here.}.
    An anti-invariant Higgs bundle $(E, \varphi, \psi)$ is \textbf{semistable} if and only if for every $\sigma$-isotropic subbundle $F$, we have that 
    \begin{equation*}
        \text{deg}(F) \leq 0.
    \end{equation*}
\end{definition}

In \cite[Proposition 4.2]{zelaci2019moduli} it is shown that an anti-invariant bundle $(E, \psi)$ is semistable if and only if it is stable as a vector bundle. 

\begin{proposition}\label{lem:antiinvstab}
    An anti-invariant Higgs bundle $(E, \varphi, \psi)$ is semistable if and only if the underlying Higgs bundle $(E, \varphi)$ is semistable. 
\end{proposition}
\begin{proof}
    The proof is analogous to the proof of \cite[Proposition 4.2]{zelaci2019moduli}. The if direction is clear. Suppose that $E$ is semistable as an anti-invariant Higgs bundle. If $F$ is any Higgs subbundle of $(E, \varphi)$, then $F^{\perp_\sigma}$ is also a Higgs subbundle. It follows that the bundle $N$ generated by $F \cap F^{\perp_\sigma}$ is also a Higgs subbundle, which is $\sigma$-isotropic by construction. In loc. cit. it is shown that $\text{deg}(N) = \text{deg}(F)$, so it follows that $\text{deg}(F) \leq 0$ because $E$ is semistable as an anti-invariant Higgs bundle. It follows that $E$ is also semistable as a Higgs bundle.
\end{proof}

\subsection{Stratification}

In fact, the locus of semistable anti-invariant Higgs bundles is just one part of a bigger structure, called a $\Theta$-stratification. These stratifications were introduced by Halpern-Leistner in \cite{halpern2014structure}, and form a generalisation of the HKKN stratification in GIT and the Harder-Narasimhan stratification for the stack of vector bundles. Here, we follow \cite[Definition 6.1]{alper2023existence}. We denote by $\text{Filt}(\mathcal{X})$ the stack of morphisms from $\Theta$ to $\mathcal{X}$. This is an algebraic stack by \cite[Proposition 1.1.2]{halpern2014structure}. There is a natural morphism $\text{ev}_1 \colon \text{Filt}(\mathcal{X}) \rightarrow \mathcal{X}$ given by sending a morphism $f \colon \Theta \rightarrow \mathcal{X}$ to $f(1)$.

\begin{definition}\cite[Definition 6.1]{alper2023existence}
    Let $\mathcal{X}$ be an algebraic stack over a locally Noetherian algebraic space $S$. 
    \begin{enumerate}
        \item A \textbf{$\Theta$-stratum} in $\mathcal{X}$ consists of a union of connected components $\mathcal{Z}^+ \subset \text{Filt}(\mathcal{X})$ such that $\text{ev}_1 \colon \mathcal{Z}^+ \rightarrow \mathcal{X}$ is a closed immersion. 
        \item A \textbf{$\Theta$-stratification} indexed by a totally ordered set $\Gamma$ is a cover of $\mathcal{X}$ by open substacks $\mathcal{X}_{\leq c}$ for $c \in \Gamma$ such that $\mathcal{X}_{\leq c} \subset \mathcal{X}_{\leq c'}$ for $c \leq c'$, along with a $\Theta$-stratum $\mathcal{Z}_c^+$ in $\mathcal{X}_{\leq c}$ such that the complement of the image in $\mathcal{X}_{\leq c}$ is 
        \begin{equation*}
            \mathcal{X}_{<c} := \bigcup\limits_{c\sp{\prime} < c} \mathcal{X}_{\leq c\sp{\prime}}.
        \end{equation*}
        We require that for all $x \in \mathcal{X}$, the set $\lbrace c \in \Gamma \, | \, x \in \mathcal{X}_{\leq c}$ has a minimal element. We assume that $\Gamma$ has a minimal element $0 \in \Gamma$.
        \item We say that a $\Theta$-stratification is \textbf{well-ordered} if for any $x \in \mathcal{X}$ the totally ordered set $\lbrace c \in \Gamma \, | \, \text{ev}_1(\mathcal{Z}_c^+) \cap \overline{\lbrace x \rbrace} \neq \emptyset \rbrace$ is well-ordered.
    \end{enumerate}
\end{definition}

The open substack $\mathcal{X}^{ss} := \mathcal{X}_{\leq 0}$ is called the \textbf{semistable locus}. For any $x \in \mathcal{X} / \mathcal{X}^{ss}$, it follows from the definition that there is a unique $c$ such that $x \in \text{ev}_1(\mathcal{Z}_c^+)$. This gives rise to a canonical map $f: \Theta \rightarrow \mathcal{X}$ with $f(1) \cong x$, this is called an \textbf{HN filtration}. Our goal in this section is to show that the stack of anti-invariant Higgs bundles admits a well-ordered $\Theta$-stratification such that the semistable locus agrees with our earlier definition of semistable anti-invariant Higgs bundles. We will show this for $\sigma$-symmetric bundles, the case of $\sigma$-alternating is completely analogous.

One way of constructing $\Theta$-stratifications is through standard numerical invariants, as explained in \cite[Section 4]{halpern2014structure}. We will not explain the theory of numerical invariants in general, but explain a natural standard numerical invariant on the stack of anti-invariant Higgs bundles and how it gives rise to a $\Theta$-stratification. 

Consider a morphism $ f \colon \Theta \rightarrow \Bunsl$. As explained in \cite[Lemma 1.13]{heinloth2018hilbert}, such a morphism is induced by a cocharacter $\lambda_f \colon \Gm \rightarrow \text{SL}(n)$. After fixing a conjugation invariant norm $\lVert \cdot \rVert$, we can define a locally constant function

\begin{equation*}
    \mu \colon  \lvert \text{Filt}(\HBunantisym) \rvert \rightarrow \mathbb{R} \colon f \mapsto \frac{\wt f^* \Ldet}{\lVert \lambda_f \rVert},
\end{equation*}
where $\lambda_f$ is the cocharacter induced by the composition $f \colon \Theta \rightarrow \HBunantisym$.

Using the function $\mu$, we can define a function
\begin{equation*}
\begin{split}
    M^\mu \colon \HBunantisym \rightarrow \mathbb{R}_{\geq 0} \colon
\\
    x \mapsto \max{\lbrace 0 \rbrace \cup \sup \lbrace \mu(f) \, | \, f \in \text{Filt}(\HBunantisym) \text{ such that } f(1) \cong x \rbrace}
\end{split}
\end{equation*}

For $c \in \mathbb{R}_{\geq 0}$, the corresponding open substack of $\HBunantisym$ will be given by 
\begin{equation*}
    {\text{HBun}^{\sigma, +}_{\text{SL}_n(C)}}_{\leq c}  = \lbrace x \in \HBunantisym \, | \, M^\mu(x) \leq c \rbrace.
\end{equation*}

To verify that this forms a $\Theta$-stratification, we will use the following theorem. 

\begin{theorem}\cite[Theorem 4.5.1]{halpern2014structure} \label{theorem:stratcriteria}
    The numerical invariant $\mu$ defines a $\Theta$-stratification if and only if it satisfies the following two properties: 
    \begin{enumerate}
        \item \textbf{Simplified HN-Specialization Property}: For any essentially of finite type DVR $R$ with fraction field $K$ and residue field $\kappa$, and any map $\xi \colon \Spec R \rightarrow \HBunantisym$ whose generic point is unstable and a HN filtration $f_K$ (a filtration $f$ such that $\mu(f) = M^\mu(f(1))$), one has 
        \begin{equation*}
            \mu(f_K) \leq M^\mu(\xi_{\Spec \kappa}),
        \end{equation*}
        and when equality holds there is an extension of DVRs $R \subset R'$ with fraction fields $K \subset K'$ such that $f_K|_{K'}$ extends to $\Spec R'$.
        \item \textbf{HN-Boundedness}: For any map $\xi \colon T \rightarrow \HBunantisym$ from a finite type affine scheme, there exists a quasicompact substack $\mathcal{X}' \subset \HBunantisym$ such that for all finite type points $p \in T(k)$, and filtrations $f$ of $p$, there is another filtration $f'$ with $\mu(f) \leq \mu(f')$ and $f'(0) \in \mathcal{X}'$.
    \end{enumerate}
\end{theorem}

\begin{remark}
    The theorem as stated in \cite{halpern2014structure} contains some additional hypotheses on $\mu$. However, our numerical invariant is of the form described in Definition 4.1.14 of loc. cit. and by Example 4.5.2 of loc. cit. it follows that these additional hypotheses are automatically satisfied.
\end{remark}
The rest of this section will be dedicated to verifying the two properties in this theorem. We will first verify the existence and uniqueness of HN filtrations. The following result is completely analogous to the case of ordinary vector bundles. 

\begin{lemma}\label{lem:Hnfiltcriteria}
    Let $(E, \varphi, \psi)$ be an unstable $\sigma$-symmetric anti-invariant Higgs bundle and let 
    \begin{equation*}
        0 = F^0 \subset F^1 \subset \ldots \subset F^k \subseteq {F^k}^{\perp_\sigma} \subset \ldots {F^1}^{\perp_\sigma} \subset {F^0}^{\perp_\sigma} = E.
    \end{equation*}
    be a filtration by Higgs subbundles denoted by $f$. Then this filtration maximises $\mu(f) = M^\mu((E, \varphi, \psi))$ if and only if
    \begin{itemize}
        \item $\mu(E/F^k) < \mu(F^k/F^{k-1}) < \ldots < \mu(F_2/F_1) < \mu(F_1)$,
        \item and $F^j/F^{j-1} \oplus {F^{j-1}}^{\perp_\sigma}/{F^j}^{\perp_\sigma}$ for $j=1, \ldots k$, and ${F^k}^{\perp_\sigma} / F^k$ are semistable $\sigma$-symmetric anti-invariant Higgs bundles.
    \end{itemize}
\end{lemma}

Using this proposition, we can prove completely analogously to the case of ordinary vector bundles \cite[Proposition 1.3.9]{harder1975cohomology} and orthogonal/symplectic bundles \cite[Proposition 3.2]{choe2022minimal} that HN filtrations exist and are unique. 
\begin{proposition}
     Let $(E, \varphi, \psi)$ be an unstable $\sigma$-symmetric anti-invariant Higgs bundle. Then there exists a unique HN filtration of $(E, \varphi, \psi)$.
\end{proposition}
\begin{proof}
    We let $E_1 \subset E$ be a $\sigma$-isotropic Higgs subbundle such that 
    the slope of $E_1$ is maximal among the $\sigma$-isotropic Higgs subbundles of $E$, and the rank of $E_1$ is maximal among those subbundles maximising the slope. We can then continue this process for subbundles containing $E_1$. Like this we arrive at a filtration by  $\sigma$-isotropic Higgs subbundles
    \begin{equation*}
        0 \subset E_1 \subset E_2 \subset \ldots E_{k} \subset E
    \end{equation*}
    such that successive quotients are semistable Higgs bundles, their degrees are strictly decreasing and if $F$ is a $\sigma$-isotropic subbundle containing $E_k$, then $\mu(F/E_k) \leq \mu(E/E_k)$. The last point immediately implies that ${E_k}^{\perp_\sigma}/E_k$ is a semistable $\sigma$-symmetric anti-invariant Higgs bundle. For $j \neq k$, the bundle $E_j / E_{j-1} \oplus {E_{j-1}}^{\perp_\sigma} / {E_j}^{\perp_\sigma}$ is of the form $F \oplus \sigma^* F^*$, and this is semistable as an anti-invariant Higgs bundle if and only if $F$ is semistable as a Higgs bundle. We thus obtain a filtration of the form
    \begin{equation*}
        0 \subset E_1  \subset \ldots E_{k} \subset {E_k}^{\perp_\sigma} \ldots \subset {E_1}^{\perp_\sigma} \subset E,
    \end{equation*}
    satisfying the criteria of Lemma \ref{lem:Hnfiltcriteria}.
    Because all quotients are semistable Higgs bundles and the slopes are decreasing, it immediately follows that this filtration is also the Harder-Narasimhan filtration of the Higgs bundle $(E, \varphi)$. This proves the uniqueness.
\end{proof}

Once we know that Harder-Narasimhan filtrations exist and are unique, one of the two criteria for the existence of a $\Theta$-stratification is immediate. 
\begin{proposition}\label{prop:HNbounded}
    The stack $\HBunantisym$ satisfies HN Boundedness. 
\end{proposition}
\begin{proof}
    Suppose $\xi : T \rightarrow \HBunantisym$ is a map from a finite type affine scheme. We know that only finitely many HN types of Higgs bundles occur in this family.  Moreover, for a fixed HN type $\tau$ the stack parametrising graded Higgs bundles of this HN type is a product of stacks of semistable Higgs bundles, which are known to be quasicompact, and thus its preimage in $\HBunantisym$ is as well. Let $\mathcal{X}'$ be the union of these finitely many quasicompact stacks. For a finite type point $p \in T(k)$, if $f'$ is the HN filtration of $\xi(p)$, then it follows by definition that for any other filtration $f$ of $\xi(p)$, we have $\mu(f) \leq \mu(f')$, and $f'(0) \in \mathcal{X}'$ by construction.
\end{proof}

The HN-specialization property can be deduced from that of ordinary bundles. 

\begin{proposition}\label{prop:HNspec}
    The stack $\HBunantisym$ satisfies the simplified HN-Specialisation property.    
\end{proposition}
\begin{proof}
    Let $\xi : \Spec R \rightarrow \HBunantisym$ be a morphism from the spectrum of a DVR. In \cite[Theorem 2]{shatz1977decomposition}, it is shown that if $E$ is a family of vector bundles on $C$ over $\Spec R$, then any filtration over the generic point extends uniquely to the whole of $\Spec R$. Because the conditions of being a Higgs subbundle and being $\sigma$-isotropic are closed, it follows that if the filtration over the generic point was of the type in Proposition \ref{prop:filtonanti-invHiggs}, then the filtration over $\Spec R$ is as well. It follows that the HN filtration $f_K$ extends to a filtration of $\xi_{\Spec \kappa}$ and this implies that
    \begin{equation*}
        \mu(f_K) \leq M^\mu(\xi_{\Spec \kappa}).
    \end{equation*}
\end{proof}

Putting this all together, we obtain the following theorem.

\begin{theorem}\label{theorem:antiinvHiggsstrat}
    The stack $\HBunantisym$ admits a well-ordered $\Theta$-stratification such that the semistable locus is quasicompact. 
\end{theorem}
\begin{proof}
    Because of Proposition \ref{prop:HNbounded} and Proposition \ref{prop:HNspec}, Theorem \ref{theorem:stratcriteria} now implies that the stack admits a $\Theta$-stratification. The well-ordering is analogous to the proof of well-ordering in \cite[Lemma 8.2
    ]{alper2023existence}. 
    Finally suppose $(E, \varphi, \psi)$ is a semistable anti-invariant bundle. By analogy with \cite[Corollary 3.3]{nitsure1991moduli} it follows that there are only finitely many Harder-Narasimhan types for the underlying anti-invariant bundle $(E, \psi)$. Denote $\mathcal{X}' \subset \Bunantisym$ for the substack parametrising those Harder-Narasimhan types. Because $\Bunantisym$ is isomorphic to a stack of parahoric torsors, it follows by the proof of  \cite[Proposition 3.18]{heinloth2018hilbert} that $\mathcal{X'}$ is quaiscompact. Thus the preimage under the affine morphism $\HBunantisym \rightarrow \Bunantisym$ is as well and this contains the semistable locus by construction.
\end{proof}

\section{Good Moduli Spaces} 

\subsection{Good Moduli Spaces}

Good moduli spaces were introduced in \cite{alper2013good} as an approximation of an algebraic stack by an algebraic space generalising the notion of good quotient. 

\begin{definition}
    A quasicompact and quasiseparated morphism $\pi : \mathfrak{X} \rightarrow X$ from an algebraic stack $\mathfrak{X}$ to an algebraic space $X$ is a \textbf{good moduli space} if 
    \begin{enumerate}
        \item $\mathcal{O}_X \rightarrow \pi_* \mathcal{O}_\mathfrak{X}$ is an isomorphism, and
        \item  $\pi_* \colon\QCoh(\mathfrak{X}) \rightarrow \QCoh(X)$ is exact.
    \end{enumerate}
\end{definition}
\begin{remark}
    \begin{enumerate}
        \item A morphism satisfying the second property in this definition is said to be cohomologically affine.
        \item In \cite[Theorem 4.16]{alper2013good} it is shown (along with other useful properties) that $\pi$ is surjective, universally closed, universal for maps to schemes and $\pi(x) = \pi(y)$ if and only if $\overline{\lbrace x \rbrace} \cap \overline{\lbrace y \rbrace} \neq \emptyset$.
    \end{enumerate}
\end{remark}

\subsection{Criteria}

In \cite{alper2023existence}, criteria are given for when an algebraic stack admits a good moduli space. 

\begin{theorem}\cite[Theorem A]{alper2023existence}
    Let $\mathfrak{X}$ be an algebraic stack of finite type with affine diagonal. Then $\mathfrak{X}$ admits a separated good moduli space if and only if it is $\Theta$-complete and S-complete.
\end{theorem}
\begin{remark}
    We know that $\HBunantisymss$ is locally of finite type with affine diagonal because it is an open substack of $\HBunantisym$, and Theorem \ref{theorem:antiinvHiggsstrat} shows that it is quasicompact. Therefore, it satisfies the hypotheses of this theorem. 
\end{remark}

The notions of $\Theta$- and S-completeness are both filling conditions with respect to codimension 2 points of algebraic stacks over discrete valuation rings. 

For $\Theta$-completeness, we consider the stack 
\begin{equation*}
    \Theta_R = \Theta \times \Spec R = [\Spec R[t] / \Gm ]
\end{equation*}
where $R$ is a DVR. This has a codimension 2 point $0$ defined by the vanishing of $t$ and a uniformizer of $R$. 

\begin{definition}
    Let $\mathfrak{X}$ be an algebraic stack. We say that $\mathfrak{X}$ is \textbf{$\Theta$-complete} \footnote{In \cite{alper2023existence} this is called $\Theta$-reductive.} if for all DVRs $R$ and every diagram 
    % https://q.uiver.app/#q=WzAsMyxbMCwwLCJcXFRoZXRhX1IgLyAwIl0sWzIsMCwiXFxtYXRoZnJha3tYfSJdLFswLDIsIlxcVGhldGFfUiJdLFswLDIsIiIsMCx7InN0eWxlIjp7InRhaWwiOnsibmFtZSI6Imhvb2siLCJzaWRlIjoidG9wIn19fV0sWzAsMV0sWzIsMSwiXFxleGlzdHMgISIsMCx7InN0eWxlIjp7ImJvZHkiOnsibmFtZSI6ImRhc2hlZCJ9fX1dXQ==
\[\begin{tikzcd}
	{\Theta_R / 0} && {\mathfrak{X}} \\
	\\
	{\Theta_R}
	\arrow[hook, from=1-1, to=3-1]
	\arrow[from=1-1, to=1-3]
	\arrow["{\exists !}", dashed, from=3-1, to=1-3]
\end{tikzcd}\]
    there exists a unique map $\Theta_R \rightarrow \mathfrak{X}$ filling in the diagram.
\end{definition}

For the stacks under consideration here, whose elements are sheaves on $C$ with some decoration, there is a nice interpretation of this criterion. We have 
\begin{equation*}
    \Theta_R / 0 = \Spec R \bigcup_{\Spec K} \Theta_K,
\end{equation*}
where $K$ is the fraction field of $R$. Thus a morphism from $\Theta_R \backslash 0$ consists of a sheaf $\mathcal{E}$ with these decorations on $C_R$ and a morphism from $\Theta_K$ such that the image of 1 is the restriction of $\mathcal{E}$ to $C_K$, or equivalently a filtration of $\mathcal{E}_K$ satisfying some extra conditions determined by the decorations as in Section 4.1. Being $\Theta$-complete then requires that this filtration extends uniquely to $\mathcal{E}$.

For S-completeness, the stack under consideration is 
\begin{equation*}
    \Phi_R = [(\Spec R[x,y]_{1,-1} / (xy - \pi)) / \Gm ]
\end{equation*}
where $\pi$ is a uniformiser of $R$ and $\Gm$ acts with weights 1 and -1 on $x$ and $y$ respectively. This again has a codimension 2 point defined by the vanishing of $x$ and $y$. 

\begin{definition}
    Let $\mathfrak{X}$ be an algebraic stack. We say that $\mathfrak{X}$ is \textbf{S-complete} if for all DVRs $R$ and every diagram 
    % https://q.uiver.app/#q=WzAsMyxbMCwwLCJcXFRoZXRhX1IgLyAwIl0sWzIsMCwiXFxtYXRoZnJha3tYfSJdLFswLDIsIlxcVGhldGFfUiJdLFswLDIsIiIsMCx7InN0eWxlIjp7InRhaWwiOnsibmFtZSI6Imhvb2siLCJzaWRlIjoidG9wIn19fV0sWzAsMV0sWzIsMSwiXFxleGlzdHMgISIsMCx7InN0eWxlIjp7ImJvZHkiOnsibmFtZSI6ImRhc2hlZCJ9fX1dXQ==
\[\begin{tikzcd}
	{\Phi_R / 0} && {\mathfrak{X}} \\
	\\
	{\Phi_R}
	\arrow[hook, from=1-1, to=3-1]
	\arrow[from=1-1, to=1-3]
	\arrow["{\exists !}", dashed, from=3-1, to=1-3]
\end{tikzcd}\]
    there exists a unique map $\Phi_R \rightarrow \mathfrak{X}$ filling in the diagram.
\end{definition}

We have that

\begin{equation*}
    \Phi_R / 0 = \Spec R \bigcup_{\Spec K} \Spec R,
\end{equation*}
so this can be interpreted as a weakened separatedness condition.

In \cite[Proposition 6.8.29]{alpernotes} it is shown that the stack of coherent sheaves $\Coh$ is both $\Theta$- and S-complete. Furthermore, \cite[Propositions 3.21 and 3.44]{alper2023existence} show that if $\mathfrak{X} \rightarrow \mathfrak{Y}$ is an affine morphism of algebraic stacks and $\mathfrak{Y}$ is $\Theta$- and/or S-complete, then so is $\mathfrak{X}$. Together with Proposition \ref{prop:Hcohaff}, this implies the following result.

\begin{proposition}\label{prop:HCoHcrit}
    The stack $\HCoh$ is $\Theta$- and S-complete.
\end{proposition}

This stack will not admit a good moduli space, because it is not quasicompact. 

\section{Moduli Space of Semistable Higgs Bundles}

In this section we show that the stack of semistable Higgs bundles is $\Theta$- and S-complete and thus admits a good moduli space. Because there is already a GIT construction of this moduli space due to Nitsure in \cite{nitsure1991moduli} and Simpson in \cite{Simpson1} \cite{Simpson2}, it is expected that this is true, but we could not find a proof in the literature and this is a necessary component of the proof for the anti-invariant case. 

The idea is that in both the criteria, we already know that there exists a unique extension to the bigger stack $\HCoh$. The only thing left to verify then is that the image of $0$ is in the substack of semistable Higgs bundles. The arguments are analogous to those for the stack of semistable vector bundles on a curve as presented in \cite{alper2022projectivity}.

\begin{proposition}
    $\HBunrdss$ is $\Theta$-complete
\end{proposition}
\begin{proof}
    Given a morphism $f: \Theta_R / 0 \rightarrow \HBunrdss$, we need to show that it extends uniquely to a morphism from $\Theta_R$. When we compose with the inclusion $\HBunrdss \rightarrow \HCoh$, then a unique extension $g: \Theta_R \rightarrow \HCoh$ exists because by Proposition \ref{prop:HCoHcrit}, $\HCoh$ is $\Theta$-complete. 

    % https://q.uiver.app/#q=WzAsNCxbMCwwLCJcXFRoZXRhX1IgLyAwIl0sWzAsMiwiXFxUaGV0YV9SIl0sWzIsMCwiXFxIQnVucmRzcyJdLFsyLDIsIlxcSENvaCJdLFswLDIsImYiLDJdLFswLDEsIiIsMix7InN0eWxlIjp7InRhaWwiOnsibmFtZSI6Imhvb2siLCJzaWRlIjoidG9wIn19fV0sWzIsMywiIiwwLHsic3R5bGUiOnsidGFpbCI6eyJuYW1lIjoiaG9vayIsInNpZGUiOiJ0b3AifX19XSxbMSwzLCJnIiwyXV0=
\[\begin{tikzcd}
	{\Theta_R / 0} && \HBunrdss \\
	\\
	{\Theta_R} && \HCoh
	\arrow["f"', from=1-1, to=1-3]
	\arrow[hook, from=1-1, to=3-1]
	\arrow[hook, from=1-3, to=3-3]
	\arrow["g"', from=3-1, to=3-3]
\end{tikzcd}\]
To get a unique lift in the above diagram, it then suffices to show that if $g: \Theta_R \rightarrow HCoh$ is a morphism such that the restriction to $\Theta_R / 0$ lands in $\HBunrdss$, then $g(0) \in \HBunrdss$. 

A morphism $\Theta_R / 0$ corresponds to a semistable Higgs bundle $E$ on $C_R$ of rank $r$ and degree $d$ and a Higgs filtration 
\begin{equation}
    \lbrace 0 \rbrace = F_0 \subset F_1 \subset \ldots \subset F_{n-1} \subset F_n = E_K
\end{equation}
of the restriction to the generic point of the DVR such that $gr(E_K)$ is semistable of rank $r$ and degree $d$.

The fact that this morphism extends in $\HCoh$ corresponds to the fact that the filtration of $E_K$ extends globally as a filtration by Higgs subsheaves
\begin{equation}
    \lbrace 0 \rbrace = E_0 \subset E_1 \subset \ldots \subset E_{n-1} \subset E_n = E.
\end{equation}

We first show that each $E_i$ for $i \in \lbrace 1, \ldots, n \rbrace$ is a semistable Higgs bundle of slope $ \mu := \frac{r}{d}$. Because we know that $gr(E_K)$ is semistable of slope $\mu$, it follows that for each $i$, $F_i / F_{i-1}$ is semistable of slope $\mu$. Because $F_n/F_{n-1}$ and $F_n$ both have slope $\mu$ and in an extension $0 \rightarrow V_1 \rightarrow V_2 \rightarrow V_3 \rightarrow 0$, the slope of $V_2$ is a (non-trivial) convex combination of those of $V_1$ and $V_3$, it must hold that $F_{n-1}$ has slope $\mu$. By induction we can similarly show that this holds for the other non-trivial subbundles in the filtration of $E_K$. Because the slope is constant in flat families, it follows that the slope of the non-trivial $E_i$ is also $\mu$.

Because all of the $E_i$ are now subsheaves of $E$ of the same slope, they must also be semistable Higgs bundles. Over the residue field $\kappa$ of the DVR, we now have a filtration
\begin{equation}
    \lbrace 0 \rbrace = E_{0,\kappa} \subset E_{1,\kappa} \subset \ldots \subset E_{n-1, \kappa} \subset E_{n, \kappa} = E
\end{equation}
where each nontrivial $E_{i,\kappa}$ is a semistable Higgs bundle of slope $\mu$. To conclude that $E_{i,\kappa}/ E_{i-1, \kappa}$ is semistable of slope $\mu$, we can apply Lemma \ref{lem:quotss}.

This shows that $gr(E_\kappa)$ is semistable of slope $\mu$, which is precisely the statement that $g(0) \in \HBunrdss$.
\end{proof}
\begin{lemma}\label{lem:quotss}
    Suppose $E$ is a semistable Higgs bundle of slope $\mu$ and $F$ is a non-trivial semistable Higgs subbundle of the same slope. Then $E/F$ is also semistable of slope $\mu$.
\end{lemma}
\begin{proof}
    Because $E$ is an extension of $F$ and $E/F$, it follows that the slope of $E$ is a (non-trivial) convex combination of the slopes of $F$ and $E/F$. This forces the slope of $E/F$ to be $\mu$. 
    Now suppose $H$ is another Higgs subbundle containing $F$, so that $H/F$ is a Higgs subbundle of $E/F$. As $E$ is semistable, the slope of $H$ is less than or equal to $\mu$. Because it is again a non-trivial convex combination of the slopes of $F$ and $H/F$, this implies that the slope of $H/F$ has to be less than or equal to $\mu$. This proves that $E/F$ is a semistable Higgs sheaf and thus a semistable Higgs bundle.
\end{proof}

We will prove that this stack also satisfies S-completeness by analysing opposite filtrations of semistable Higgs bundles. 

\begin{definition}
    Two $\Z$-graded filtrations 
    \begin{equation*}
        E_\bullet \colon \lbrace 0 \rbrace \subset \ldots E_{i-1} \subset E_{i} \subset E_{i + 1} \subset \ldots E
    \end{equation*}
    \begin{equation*}
        F_\bullet \colon F \supset \ldots F_{i-1} \supset F_{i} \supset F_{i + 1} \supset \ldots \lbrace 0 \rbrace
    \end{equation*}
    are \textbf{opposite} if 
    \begin{equation}
        E_i / E_{i-1} \cong F_i / F_{i + 1}
    \end{equation}
    for all $i \in \Z$.
\end{definition}

In \cite[Proposition 6.8.32]{alpernotes} it is shown that if $\mathcal{U} \subset \Coh$ is an open substack and for all opposite filtrations $E_\bullet, F_\bullet$ of $E,F \in \mathcal{U}(k)$ we have that $gr(E_\bullet) \in \mathcal{U}(k)$, then $\mathcal{U}$ is S-complete. The argument is that any morphism $\phi_R \backslash 0 \rightarrow \mathcal{U}$ extends to a morphism $ g \colon \phi_R \rightarrow \Coh$. Restricting this to the locus where $\pi = 0$, for $\pi$ some uniformiser of $R$, we get a morphism $ g' \colon [\Spec k[x,y]/(xy) / \Gm] \rightarrow \Coh$. To check that $g(0) \in \mathcal{U}$, it suffices to check that the $g'(0) \in \mathcal{U}$. The morphism $g'$ corresponds precisely to two opposite filtrations and $g'(0)$ is the associated graded of these filtrations.

If we want a similar criterion for open substacks of $\HCoh$, the only thing that changes is that we should consider opposite Higgs filtrations. This proves the following lemma.

\begin{lemma}\label{lem:Scomp=oppfilt}
    Suppose $\mathcal{U} \subset \HCoh$ is an open substack and for all opposite Higgs filtrations $E_\bullet, F_\bullet$ of $E,F \in \mathcal{U}(k)$ we have that $gr(E_\bullet) \in \mathcal{U}(k)$, then $\mathcal{U}$ is S-complete.
\end{lemma}

Opposite filtrations of a semistable Higgs bundle imply a lot of structure for the filtrations. This is captured in the following lemma.

\begin{lemma}\label{lem:oppfiltslope}
    Let $E_\bullet$, $F_\bullet$ be opposite Higgs filtrations of semistable Higgs bundles $E,F$ of rank $r$ and degree $d$. Then all of the non-trivial $E_i (\text{and } E_i/E_{i-1}),F_i$ are semistable Higgs subbundles of $E,F$ of slope $\mu := r/d$.
\end{lemma}
\begin{proof}
    Suppose that $E_i = \lbrace 0 \rbrace$ for $i \leq 0$ (and thus $F_i = F$ for $i \leq 1$) and $E_i = E$ for $i \geq n$ (and thus $F_i = \lbrace 0 \rbrace$ for $ i \geq n + 1$). We show that $E_1$ and $F_2$ are semistable of slope $\mu$. Because $E_1$ is a Higgs subbundle of $E$, we have that
    \begin{equation*}
        \mu(E_1) \leq \mu.
    \end{equation*}
    But because the filtrations are opposite, we also know that 
    \begin{equation*}
        E_1 \cong F / F_2,
    \end{equation*}
    so because it is isomorphic to a Higgs quotient of $F$, we have that 
    \begin{equation*}
        \mu(E_1) \geq \mu.
    \end{equation*}
    We conclude that $\mu(E_1) = \mu.$ As $E_1 \cong F / F_2$, this also implies that $\mu(F_2) = \mu$. Because they are both Higgs subbundles of semistable Higgs bundles of slope $\mu$, it follows that they are semistable.

    The two filtrations 
    \begin{equation*}
        \lbrace 0 \rbrace \subset E_2 / E_1 \subset \ldots \subset E_{n-1}/E_1 \subset E/E_1,
    \end{equation*}
    \begin{equation}
        F_2 \supset F_3 \supset \ldots \supset F_n \supset \lbrace 0 \rbrace
    \end{equation}
    are now opposite filtrations of the semistable Higgs bundles $E/E_1$ and $F_2$. We can argue by induction that $F_i$ is a semistable Higgs bundle of slope $\mu$ for each $i$ and similarly that $E_i/E_{i -1}$ is. This will in turn imply that $E_i$ is a semistable Higgs bundle for each $i$.
\end{proof}

Combining the previous lemmas, we can show that the stack of Higgs bundles is S-complete.

\begin{proposition}
    $\HBunrdss$ is S-complete.
\end{proposition}
\begin{proof}
    By Lemma \ref{lem:Scomp=oppfilt} it suffices to show that given two opposite filtrations of semistable Higgs bundles of rank $r$ and degree $d$, the associated graded is also semistable of rank $r$ and degree $d$. 
    
    By Lemma \ref{lem:oppfiltslope}, each of the summands of the associated graded is semistable of slope $\mu := r/d$. This implies that the whole associated graded is semistable of rank $r$ and degree $d$.
\end{proof}

Together with the $\Theta$-completeness, this establishes the following theorem.

\begin{theorem}\label{th:gmsH}
    The stack $\HBunrdss$ admits a good moduli space.
\end{theorem}

Analogously, we can show that the stack of semistable $\text{SL}_n$-bundles $\HBunslss$ admits a good moduli space. The only added difficulty here consists of proving that the trivialisation of the determinant also extends over $\Theta_R$ and $\Phi_R$ in both criteria. This is done analogously to the extension of the isomorphism in the proof of Theorem \ref{th:gmsexists}
\section{Moduli Space of (anti-invariant) Higgs bundles}

Using the established results for the stack of Higgs bundles, we can show that the stacks of anti-invariant Higgs bundles admit good moduli spaces. Because of Lemma \ref{lem:antiinvstab}, we have forgetful morphisms. 
\begin{equation*}
    \HBunantisymss \rightarrow \HBunslss,
\end{equation*}
and 
\begin{equation}
    \HBunantialtss \rightarrow \HBunslss.
\end{equation}

We already know the target stack admits a good moduli space and is thus $\Theta$- and S-complete. We will show that these properties lift along the forgetful morphism. 

\begin{theorem}\label{th:gmsexists}
    The stacks $\HBunantisymss$ and $\HBunantialtss$ admit good moduli spaces, which we will denote by $\gmssym$ and $\gmsalt$. 
\end{theorem}
\begin{proof}
    We argue that the stack of $\sigma$-symmetric anti-invariant Higgs bundles is $\Theta$-complete. The property of S-completeness and the $\sigma$-alternating case are proven completely analogously.

    So suppose we are given a morphism
    \begin{equation*}
        \Theta_R / 0 \rightarrow \HBunantisymss.
    \end{equation*}
    We can interpret this as a family of $\sigma$-symmetric anti-invariant Higgs bundles $(\mathcal{E}', \varphi', \psi')$ on $(\Spec R[t] / 0) \times C$ equivariant with respect to the $\Gm$-action on $\Spec R[t]$. 

    Because $\HBunslss$ is $\Theta$-complete, it follows that we can extend the underlying semistable Higgs bundle $(\mathcal{E}', \varphi')$ to a $\Gm$-equivariant semistable Higgs bundle $(\mathcal{E}, \varphi)$ on $\Spec R[t] \times C$ which admits an isomorphism $\psi'$ away from $\lbrace 0 \rbrace \times C$. Such an isomorphism is a section of the vector bundle $\mathcal{E}^* \otimes \sigma^* \mathcal{E}^*$. Because $\lbrace 0 \rbrace \times C$ is codimension 2, it follows from algebraic Hartogs' lemma that this extends to an endomorphism 
    \begin{equation}
        \psi : \sigma^* \mathcal{E} \rightarrow \mathcal{E}^*.
    \end{equation}
    Locally, the locus where $\psi$ is not an isomorphism is given by the vanishing of a determinant. Thus, it follows that this locus is either codimension 1 or empty. But we already know that $\psi$ is an isomorphism outside a subset of codimension 2. It follows that $\psi$ is an isomorphism everywhere. Because it is compatible with $\varphi$ on the open subset $(\Spec R[t] / 0) \times C$, it is compatible everywhere. Lastly, it is $\Gm$-equivariant because its restriction to $(\Spec R[t] / 0) \times C$ is $\Gm$-equivariant and $0$ is a fixed point for the action. 

    The anti-invariant bundle $(\mathcal{E}, \varphi, \psi)$ on $\Theta_R \times C$ then corresponds to the unique extension
    \begin{equation*}
        \Theta_R \rightarrow \HBunantisymss.
    \end{equation*}
\end{proof}

\begin{remark}
    Completely analogously to \cite[Theorem 3.12]{alper2022projectivity}, it can be shown that these moduli spaces classify S-equivalence classes of anti-invariant Higgs bundles, for the natural generalisation of S-equivalence to this setting.
\end{remark}

\section{Hitchin system}

Composing the Hitchin system $\HBunsl \rightarrow B$ with the forgetful morphism $\HBunantisym \rightarrow \HBunsl$ (respectively $\HBunantialt \rightarrow \HBunsl$), we get the following morphisms: 

\begin{align*}
    \HBunantisym \rightarrow B \text{ and }
    \HBunantialt \rightarrow B.
\end{align*}

Because good moduli spaces are universal for maps to algebraic spaces by \cite[Theorem 6.6]{alper2013good} we get morphisms 

\begin{equation*}
    \gmssym \rightarrow B \text{ and } \gmsalt \rightarrow B.
\end{equation*}
The goal of this section is to prove that these morphisms are proper. Our strategy will be to reduce this to showing that the entire stack of anti-invariant bundles satisfies the existence part of the valuative criterion for properness. 

In \cite[Corollary 6.18]{alper2023existence} it is shown that if a stack $\mathfrak{X}$ admits a well-ordered $\Theta$-stratification such that the semistable locus $\mathfrak{X}^{ss}$ admits a separated good moduli space $M$, then a morphism $M \rightarrow B$ to a Noetherian algebraic space $B$  is proper if and only if $\mathfrak{X} \rightarrow B$ satisfies the existence part of the valuative criterion for properness. For our situation, this is summarised in the following lemma. 

\begin{lemma}\label{lem:existvalue}
    The morphism $\gmssym \rightarrow B$ (resp. $\gmsalt \rightarrow B$) is proper if and only if $\HBunantisym \rightarrow B$ (resp. $\HBunantialt \rightarrow B$) satisfies the existence part of the valuative criterion of properness. 
\end{lemma}

\begin{theorem}\label{th:Hitchinproper}
    The morphism $\gmssym \rightarrow B$ (resp. $\gmsalt \rightarrow B$) is proper. 
\end{theorem}
\begin{proof}
We show the theorem for the $\sigma$-symmetric case. The $\sigma$-alternating case is completely analogous.
    By Lemma \ref{lem:existvalue} it suffices to show that $\HBunantisym \rightarrow B$ satisfies the existence part of the valuative criterion of properness. Let $R$ be a DVR with fraction field $K$. We are looking for a lift $\Spec R \rightarrow \HBunantisym$ in the following diagram
    % https://q.uiver.app/#q=WzAsNCxbMCwwLCJcXFNwZWMgSyJdLFsyLDAsIlxcSEJ1bmFudGlzeW0iXSxbMCwyLCJcXFNwZWMgUiJdLFsyLDIsIkIiXSxbMCwxXSxbMCwyLCIiLDIseyJzdHlsZSI6eyJ0YWlsIjp7Im5hbWUiOiJob29rIiwic2lkZSI6InRvcCJ9fX1dLFsxLDNdLFsyLDNdXQ==
\[\begin{tikzcd}
	{\Spec K} && \HBunantisym \\
	\\
	{\Spec R} && B
	\arrow[from=1-1, to=1-3]
	\arrow[hook, from=1-1, to=3-1]
	\arrow[from=1-3, to=3-3]
	\arrow[from=3-1, to=3-3]
\end{tikzcd}\]
Using our various forgetful morphisms, we get a diagram 

% https://q.uiver.app/#q=WzAsNixbMCwwLCJcXFNwZWMgSyJdLFsyLDAsIlxcSEJ1bmFudGlzeW0iXSxbNSwwLCJcXEhCdW5zbCJdLFs1LDMsIkIiXSxbMiwyLCJcXEJ1bmFudGlzeW0iXSxbMCwzLCJcXFNwZWMgUiJdLFswLDFdLFsxLDJdLFsyLDNdLFsxLDRdLFswLDUsIiIsMix7InN0eWxlIjp7InRhaWwiOnsibmFtZSI6Imhvb2siLCJzaWRlIjoidG9wIn19fV0sWzUsM11d
\[\begin{tikzcd}
	{\Spec K} && \HBunantisym &&& \HBunsl \\
	\\
	&& \Bunantisym \\
	{\Spec R} &&&&& B
	\arrow[from=1-1, to=1-3]
	\arrow[from=1-3, to=1-6]
	\arrow[from=1-6, to=4-6]
	\arrow[from=1-3, to=3-3]
	\arrow[hook, from=1-1, to=4-1]
	\arrow[from=4-1, to=4-6]
\end{tikzcd}\]

In \cite[Proposition 3.3]{heinloth2018hilbert} it is shown that if $\mathcal{G}$ is a parahoric group scheme on a curve $C'$ and $\text{Bun}_\mathcal{G}$ is the stack of torsors for $\mathcal{G}$ on $C'$, then this stack satisfies the existence part of the valuative criterion. The stack $\Bunantisym$ is of this form for a certain group scheme on $C / \sigma$, so it also satisfies the existence part of the valuative criterion. So we get a lift $\Spec R \rightarrow \Bunantisym$ in the above diagram. 

Similarly, it is shown in \cite[Lemma 6.20]{alper2023existence} that the morphism $\HBunsl \rightarrow B$ satisfies the existence part of the valuative criterion. Hence, there is also a lift $\Spec R \rightarrow \HBunsl$. 

Combining these, we get a lift $\Spec R \rightarrow \HBunsl \times_{\Bunsl} \Bunantisym$. This gives us the following diagram
 
% https://q.uiver.app/#q=WzAsNSxbMCwwLCJcXFNwZWMgSyJdLFszLDAsIlxcSEJ1bmFudGlzeW0iXSxbMywyLCJcXEhCdW5zbCBcXHRpbWVzX3tcXEJ1bnNsfSBcXEJ1bmFudGlzeW0iXSxbMCw0LCJcXFNwZWMgUiJdLFszLDQsIkIiXSxbMCwxXSxbMSwyXSxbMiw0XSxbMyw0XSxbMywyXSxbMCwzXV0=
\[\begin{tikzcd}
	{\Spec K} &&& \HBunantisym \\
	\\
	&&& {\HBunsl \times_{\Bunsl} \Bunantisym} \\
	\\
	{\Spec R} &&& B. 
	\arrow[from=1-1, to=1-4]
	\arrow[from=1-4, to=3-4]
	\arrow[from=3-4, to=5-4]
	\arrow[from=5-1, to=5-4]
	\arrow[from=5-1, to=3-4]
	\arrow[from=1-1, to=5-1]
\end{tikzcd}\]

The stack $\HBunsl \times_{\Bunsl} \Bunantisym$ parametrises triples $(E, \varphi, \psi)$ consisting of a Higgs bundle $(E, \varphi)$ and an isomorphism $\psi \colon \sigma^*E \rightarrow E^*$, but they need not satisfy the compatibility condition of Equation \ref{eq:antiinvHiggs}. This compatibility condition cuts out a closed substack of $\HBunsl \times_{\Bunsl} \Bunantisym$, and the morphism $\HBunantisym \rightarrow \HBunsl \times_{\Bunsl} \Bunantisym$ is the corresponding closed immersion. It follows that the map $\Spec R \rightarrow \HBunsl \times_{\Bunsl} \Bunantisym$ lands in $\HBunantisym$ as it does so generically, and this is the desired lift of the map $\Spec R \rightarrow B$.
\end{proof}

\begin{remark}
    A description of the image of these Hitchin systems can be found in \cite[Section 5]{hacen2022hitchin}.
\end{remark}

\bibliography{References}

\newcommand{\etalchar}[1]{$^{#1}$}
\begin{thebibliography}{ABB{\etalchar{+}}22}

\bibitem[ABB{\etalchar{+}}22]{alper2022projectivity}
Jarod Alper, Pieter Belmans, Daniel Bragg, Jason Liang, and Tuomas Tajakka.
\newblock Projectivity of the moduli space of vector bundles on a curve.
\newblock In {\em Stacks {P}roject {E}xpository {C}ollection ({SPEC})}, volume
  480 of {\em London Math. Soc. Lecture Note Ser.}, pages 90--125. Cambridge
  Univ. Press, Cambridge, 2022.

\bibitem[AHLH23]{alper2023existence}
Jarod Alper, Daniel Halpern-Leistner, and Jochen Heinloth.
\newblock Existence of moduli spaces for algebraic stacks.
\newblock {\em Invent. Math.}, 234(3):949--1038, 2023.

\bibitem[Alp13]{alper2013good}
Jarod Alper.
\newblock Good moduli spaces for {A}rtin stacks.
\newblock {\em Ann. Inst. Fourier (Grenoble)}, 63(6):2349--2402, 2013.

\bibitem[Alp24]{alpernotes}
Jarod Alper.
\newblock Stacks and moduli.
\newblock \url{https://sites.math.washington.edu/~jarod/moduli.pdf}, July 5
  2024.

\bibitem[Aso06]{asok2006equivariant}
Aravind Asok.
\newblock Equivariant vector bundles on certain affine {$G$}-varieties.
\newblock {\em Pure Appl. Math. Q.}, 2(4):1085--1102, 2006.

\bibitem[BKV19]{BaragliaKamgarpourVarma}
David Baraglia, Masoud Kamgarpour, and Rohith Varma.
\newblock Complete integrability of the parahoric {H}itchin system.
\newblock {\em Int. Math. Res. Not. IMRN}, (21):6499--6528, 2019.

\bibitem[BS15]{balaji2010moduli}
V.~Balaji and C.~S. Seshadri.
\newblock Moduli of parahoric $\mathcal{G}$-torsors on a compact {R}iemann
  surface.
\newblock {\em J. Algebraic Geom.}, 24(1):1--49, 2015.

\bibitem[BT72]{BruhatTits1}
F.~Bruhat and J.~Tits.
\newblock Groupes r\'eductifs sur un corps local.
\newblock {\em Inst. Hautes \'Etudes Sci. Publ. Math.}, (41):5--251, 1972.

\bibitem[BT84]{BruhatTits2}
F.~Bruhat and J.~Tits.
\newblock Groupes r\'eductifs sur un corps local. {II}. {S}ch\'emas en groupes.
  {E}xistence d'une donn\'ee radicielle valu\'ee.
\newblock {\em Inst. Hautes \'Etudes Sci. Publ. Math.}, (60):197--376, 1984.

\bibitem[CCL22]{choe2022minimal}
Insong Choe, Kiryong Chung, and Sanghyeon Lee.
\newblock Minimal rational curves on the moduli spaces of symplectic and
  orthogonal bundles.
\newblock {\em J. Lond. Math. Soc. (2)}, 105(1):543--564, 2022.

\bibitem[CMW18]{casalaina2018introduction}
Sebastian Casalaina-Martin and Jonathan Wise.
\newblock An introduction to moduli stacks, with a view towards {H}iggs bundles
  on algebraic curves.
\newblock In {\em The geometry, topology and physics of moduli spaces of
  {H}iggs bundles}, volume~36 of {\em Lect. Notes Ser. Inst. Math. Sci. Natl.
  Univ. Singap.}, pages 199--399. World Sci. Publ., Hackensack, NJ, 2018.

\bibitem[DH23]{DamioliniHong}
Chiara Damiolini and Jiuzu Hong.
\newblock Local types of {$(\Gamma,G)$}-bundles and parahoric group schemes.
\newblock {\em Proc. Lond. Math. Soc. (3)}, 127(2):261--294, 2023.

\bibitem[Hei10]{heinloth2010uniformization}
Jochen Heinloth.
\newblock Uniformization of $\mathcal{G}$-bundles.
\newblock {\em Math. Ann.}, 347(3):499--528, 2010.

\bibitem[Hei17]{heinloth2018hilbert}
Jochen Heinloth.
\newblock Hilbert-{M}umford stability on algebraic stacks and applications to
  $\mathcal{G}$-bundles on curves.
\newblock {\em \'Epijournal G\'eom. Alg\'ebrique}, 1:Art. 11, 37, 2017.

\bibitem[Hit87a]{Hitchin1}
N.~J. Hitchin.
\newblock The self-duality equations on a {R}iemann surface.
\newblock {\em Proc. London Math. Soc. (3)}, 55(1):59--126, 1987.

\bibitem[Hit87b]{Hitchin2}
Nigel Hitchin.
\newblock Stable bundles and integrable systems.
\newblock {\em Duke Math. J.}, 54(1):91--114, 1987.

\bibitem[HL14]{halpern2014structure}
Daniel Halpern-Leistner.
\newblock On the structure of instability in moduli theory.
\newblock {\em arXiv preprint arXiv:1411.0627}, 2014.

\bibitem[HN75]{harder1975cohomology}
G.~Harder and M.~S. Narasimhan.
\newblock On the cohomology groups of moduli spaces of vector bundles on
  curves.
\newblock {\em Math. Ann.}, 212:215--248, 1974/75.

\bibitem[Kon93]{Konno}
Hiroshi Konno.
\newblock Construction of the moduli space of stable parabolic {H}iggs bundles
  on a {R}iemann surface.
\newblock {\em J. Math. Soc. Japan}, 45(2):253--276, 1993.

\bibitem[KSZ24]{kydonakissunZhao}
Georgios Kydonakis, Hao Sun, and Lutian Zhao.
\newblock Logahoric {H}iggs torsors for a complex reductive group.
\newblock {\em Math. Ann.}, 388(3):3183--3228, 2024.

\bibitem[LM10]{LogaresMartens}
Marina Logares and Johan Martens.
\newblock Moduli of parabolic {H}iggs bundles and {A}tiyah algebroids.
\newblock {\em J. Reine Angew. Math.}, 649:89--116, 2010.

\bibitem[LN08]{LaumonNgo}
G\'erard Laumon and Bao~Ch\^au Ng\^o.
\newblock Le lemme fondamental pour les groupes unitaires.
\newblock {\em Ann. of Math. (2)}, 168(2):477--573, 2008.

\bibitem[MS80]{MehtaSeshadri}
V.~B. Mehta and C.~S. Seshadri.
\newblock Moduli of vector bundles on curves with parabolic structures.
\newblock {\em Math. Ann.}, 248(3):205--239, 1980.

\bibitem[Nit91]{nitsure1991moduli}
Nitin Nitsure.
\newblock Moduli space of semistable pairs on a curve.
\newblock {\em Proc. London Math. Soc. (3)}, 62(2):275--300, 1991.

\bibitem[PR10]{pappas2010some}
Georgios Pappas and Michael Rapoport.
\newblock Some questions about $\mathcal{G}$-bundles on curves.
\newblock In {\em Algebraic and arithmetic structures of moduli spaces
  ({S}apporo 2007)}, volume~58 of {\em Adv. Stud. Pure Math.}, pages 159--171.
  Math. Soc. Japan, Tokyo, 2010.

\bibitem[Sha77]{shatz1977decomposition}
Stephen~S. Shatz.
\newblock The decomposition and specialization of algebraic families of vector
  bundles.
\newblock {\em Compositio Math.}, 35(2):163--187, 1977.

\bibitem[Sim90]{Simpson3}
Carlos~T. Simpson.
\newblock Harmonic bundles on noncompact curves.
\newblock {\em J. Amer. Math. Soc.}, 3(3):713--770, 1990.

\bibitem[Sim94a]{Simpson1}
Carlos~T. Simpson.
\newblock Moduli of representations of the fundamental group of a smooth
  projective variety. {I}.
\newblock {\em Inst. Hautes \'Etudes Sci. Publ. Math.}, (79):47--129, 1994.

\bibitem[Sim94b]{Simpson2}
Carlos~T. Simpson.
\newblock Moduli of representations of the fundamental group of a smooth
  projective variety. {II}.
\newblock {\em Inst. Hautes \'Etudes Sci. Publ. Math.}, (80):5--79, 1994.

\bibitem[Yok93]{Yokogawa}
K\^oji Yokogawa.
\newblock Compactification of moduli of parabolic sheaves and moduli of
  parabolic {H}iggs sheaves.
\newblock {\em J. Math. Kyoto Univ.}, 33(2):451--504, 1993.

\bibitem[Zel19a]{zelaci2017moduli}
Hacen Zelaci.
\newblock Moduli spaces of anti-invariant vector bundles and twisted conformal
  blocks.
\newblock {\em Math. Res. Lett.}, 26(6):1849--1875, 2019.

\bibitem[Zel19b]{zelaci2019moduli}
Hacen Zelaci.
\newblock Moduli spaces of anti-invariant vector bundles over curves.
\newblock {\em Manuscripta Math.}, 160(1-2):79--97, 2019.

\bibitem[Zel22]{hacen2022hitchin}
Hacen Zelaci.
\newblock Hitchin systems for invariant and anti-invariant vector bundles.
\newblock {\em Trans. Amer. Math. Soc.}, 375(5):3665--3711, 2022.

\end{thebibliography}
\bibliographystyle{alpha}

\end{document}